\documentclass[11pt,reqno]{amsart}

\usepackage{a4wide}
\usepackage{amsmath}
\usepackage{amsthm}
\usepackage{amsfonts}
\usepackage{amssymb}
\usepackage{xcolor}
\usepackage{mathtools}
\usepackage{eucal}
\usepackage{mathrsfs}

\usepackage{graphicx}

\usepackage{todonotes}

\theoremstyle{plain}
\newtheorem{thm}{Theorem}[section]
\newtheorem{lem}[thm]{Lemma}

\theoremstyle{definition}
\newtheorem{rem}[thm]{Remark}

\newtheorem{defi}[thm]{Definition}

\numberwithin{thm}{section}
\numberwithin{equation}{section}

\def\supp{\operatorname{supp}}
\def\esup{\operatornamewithlimits{ess\,sup}}
\def\Id{\operatorname{I}}
\def\RHS{\operatorname{RHS}}
\def\LHS{\operatorname{LHS}}

\def\Ces{\operatorname{Ces}}
\def\Cop{\operatorname{Cop}}

\def\M{\mathcal{M}}
\def\Mpl{\mathcal{M}^+}

\allowdisplaybreaks

\begin{document}

\author{Amiran Gogatishvili, Lubo\v s Pick\and Tu\u{g}\c{c}e \"{U}nver}

\title{Weighted inequalities involving Hardy and Copson operators}


\address{A. Gogatishvili,
Institute of Mathematical of the  Czech Academy of Sciences, \v Zitn\'a~25, 115~67 Praha~1, Czech Republic}
\email{gogatish@math.cas.cz}

\address{L. Pick,
Department of Mathematical Analysis, Faculty of Mathematics and Physics, Sokolovsk\'a 83, 186 75 Praha~8, Czech Republic}
\email{pick@karlin.mff.cuni.cz}

\address{T. \"{U}nver,
Institute of Mathematics of the
Czech Academy of Sciences,
\v Zitn\'a~25,
115~67 Praha~1,
Czech Republic \newline
Faculty of Science and Arts,
Kirikkale University,
71450 Yahsihan, Kirikkale,
Turkey}
\email{unver@math.cas.cz, tugceunver@kku.edu.tr}
		
\thanks{
The work was partially supported by grant no. P201-18-00580S of the Czech Science Foundation, and RVO: 67985840. The research of the A. Gogatishvili was also partially supported by Shota Rustaveli National Science Foundation (SRNSF), grant no: FR17-589. The research of L. Pick was also partially supported by the Danube Region Grant no. 8X2043 and by the grant P201/21-01976S of the Czech Science Foundation.  The research of T.~\"{U}nver was supported by the grant of The Scientific and Technological Research Council of Turkey (TUBITAK), Grant No: 1059B192000075. Part of the work on this project was carried out during the meeting \textit{Per Partes} held at Horn\'{\i} Lyse\v ciny, June 2-6, 2021.}
	
\begin{abstract}
We characterize a four-weight inequality involving the Hardy operator and the Copson operator. More precisely, given $p_1, p_2, q_1, q_2 \in (0, \infty)$, we find necessary and sufficient conditions on nonnegative measurable functions $u_1, u_2, v_1, v_2$ on $(0,\infty)$ for which there exists a positive constant $c$ such that the inequality
\begin{align*}
&\bigg(\int_0^{\infty} \bigg(\int_0^t f(s)^{p_2} v_2(s)^{p_2} ds \bigg)^{\frac{q_2}{p_2}} u_2(t)^{q_2} dt \bigg)^{\frac{1}{q_2}} \notag \\
& \hspace{3cm} \leq c  \bigg(\int_0^{\infty} \bigg(\int_t^{\infty}  f(s)^{p_1} v_1(s)^{p_1} ds \bigg)^{\frac{q_1}{p_1}} u_1(t)^{q_1} dt \bigg)^{\frac{1}{q_1}}
\end{align*}
holds for every non-negative measurable function $f$ on $(0, \infty)$. The proof is based on discretizing and antidiscretizing techniques. The principal
innovation consists in development of a new method which carefully avoids duality techniques and therefore enables us to obtain the characterization in
previously unavailable situations, solving thereby a long-standing open problem. We then apply the characterization of the inequality to the establishing of
criteria for embeddings between weighted Copson spaces $\Cop_{p_1,q_1} (u_1, v_1)$ and weighted Ces\`{a}ro spaces $\Ces_{p_2, q_2} (u_2, v_2)$, and also
between spaces $S^q(w)$ equipped with the norm
$\|f\|_{S^q(w)}= \bigg(\int_0^\infty [f^{**}(t)-f^*(t)]^q w(t)\,dt\bigg)^{{1}/{q}}$ and classical Lorentz spaces of type $\Lambda$.
\end{abstract}
	
\maketitle

\section{Introduction}

In 1981, the research of limiting properties of operators that are not bounded on $L^{\infty}$ led to the introduction in~\cite{Be-De-Sh:81} of the so-called Weak-$L^{\infty}$ space. It was determined by the norm-like functional $\sup_{t>0} (f^{**}(t)-f^{*}(t))$, in which $f^*$ is the \textit{non-increasing rearrangement} of a measurable function $f$ , defined by
\begin{equation*}
    f^*(t)=\inf\{ s>0; f_*(s) \leq t\} \quad\text{for $t\in[0,\infty)$,}
\end{equation*}
where $f_*(s)$ is the measure of the level set set $\{|f(x)| > s \}$, and $f^{**}(t)$ is the integral mean of $f^*$ over $(0,t)$. The operator $f\mapsto
f^{**}-f^{*}$ in certain (rearranged) way measures the mean oscillation of $f$, and so, quite naturally, there is an intimate connection of Weak-$L^{\infty}$ and $\operatorname{BMO}$. The above said norm-like functional cannot really be a norm, because of the intrinsic properties of the operator $f\mapsto f^{**}-f^{*}$, in particular its lack of (sub)linearity and the fact that it vanishes on constant functions. However, even so, the spaces defined in terms of this operator have proved to be very
useful. Their internal peculiarities have been dealt with using various approaches, for instance by finding equivalent expressions in terms of all kinds of Lorentz-like spaces \cite{Bo-Ma:05, Kr-Ma-Pe:00}, or working with the pointwise behaviour of $f^{**}(t)-f^{*}(t)$ \cite{Ba-Ku@86, Fr:97}. In~\cite{Ba-Mi-Ru:03}, certain combination of the methods was used in order to improve limiting Sobolev embeddings, later to be extended in~\cite{Mi-Pu:04} and~\cite{Pu:05}, and somewhat parallel to the constructions in~\cite{Ma-Pi:02}.

The impact of the functional $f^{**}-f^{*}$ to Sobolev-type inequalities and embeddings, and further to the regularity theory of PDEs, is expectable, owing to the classical estimate
\begin{equation*}
    t^{-1/n}(u^{**}(t)-{u}^{*}(t)) \lesssim (\nabla u)^{**}(t),
\end{equation*}
which holds for every smooth enough scalar function $u$ of $n$ real variables. If one wraps the term on each side of this inequality in his/her favourite norm, an embedding of some sort of a homogeneous Sobolev space into an appropriate type of the $S$-like space appears. It is therefore not surprising that, during the last four decades, function spaces involving this functional are being studied and keep on finding new applications. These include, for instance, the transfer of regularity from various
differential operators to solutions of the PDE's \cite{Mi-Pu:04,Ba-Mi-Ru:03,Ke-Pi:06}, optimal partner structures in Sobolev embeddings and for integral 
operators \cite{Ma-Pi:02,Ke-Pi:09,Ed-Mi-Mu-Pi:20,Tu:21}, various achievements in theory of function spaces and theory of interpolation such as nailing down the K\"othe duals of the classical Lorentz spaces of type Gamma \cite{Si:02}, or the hunt for a suitable environment for Calder\'on--Zygmund singular
integral operators \cite{Bo-Ma:05}. A lot more on topics related to this subject can be found in~\cite{Be-Sh:88}.

Consequently, naturally, relations to important known structures became of interest, in particular all kinds of embedding bindings to other function spaces. In~\cite{Ca-Go-Ma-Pi:08}, it was noticed that a task equivalent to characterizing an embedding between the space $S^q(u)=\{f\in \mathcal{M}:
\|f\|_{S^q(u)}<\infty\}$ and the classical Lorentz space $\Lambda^p(w)=\{f\in \mathcal{M}: \|f\|_{\Lambda^p(w)}<\infty\}$ for suitable $p,q,u,w$, in which
\begin{equation*}
\|f\|_{S^q(u)}= \bigg(\int_0^\infty \bigg[f^{**}(t)-f^*(t)\bigg]^q u(t)\,dt\bigg)^{\frac{1}{q}}
\end{equation*}
and
\begin{equation*}
\|f\|_{\Lambda^p(w)}=\bigg( \int_0^\infty f^*(t)^p w(t)\,dt\bigg)^{\frac{1}{p}},
\end{equation*}
where $f^{**}(t) = t^{-1} \int_0 ^t f^*(s) ds$ for $t\in(0,\infty)$, is to investigate the question of when a certain special case of the inequality
\begin{equation}\label{GMU-inequality}
    \begin{split}
        &\bigg(\int_0^{\infty} \bigg(\int_0^t f(s)^{p_2} v_2(s)^{p_2} ds \bigg)^{\frac{q_2}{p_2}} u_2(t)^{q_2} dt \bigg)^{\frac{1}{q_2}} 
            \\
        & \hspace{3cm} \leq c  \bigg(\int_0^{\infty} \bigg(\int_t^{\infty}  f(s)^{p_1} v_1(s)^{p_1} ds \bigg)^{\frac{q_1}{p_1}} u_1(t)^{q_1} dt \bigg)^{\frac{1}{q_1}}
    \end{split}
\end{equation}
holds for every measurable nonnegative function $f$. The main achievement of this equivalence is the reduction of a problem involving non-increasing functions to one involving all nonnegative measurable ones. When it comes to details, this is a big step forward. So, the new task at hand is to characterize for which parameters and weight functions the inequality~\eqref{GMU-inequality} holds. Particular results were obtained already in~\cite{Ca-Go-Ma-Pi:08}, but restricted to special cases of the parameter $q$, leaving the question of a complete solution open. Let us recall that the space $\Lambda^p(w)$ was introduced in~\cite{Lo:51}, while the space $S^q(u)$ in~\cite{Ca-Go-Ma-Pi:05}. Again, these spaces have been intensively studied and it would be impossible to quote all relevant references. Yet more connections, applications and references can be found in~\cite{Ca-Go-Ma-Pi:05}.

In this paper we will tackle (and solve) the above-mentioned open problem. Let us recall that it is clear from the beginning that the approach of~\cite{Ca-Go-Ma-Pi:08} cannot be used. This is caused by the fact that the duality techniques, which constitute the core of the method used there, work only for $q\ge1$. So, the challenge is not only to find necessary and sufficient conditions under which the inequality~\eqref{GMU-inequality} holds in general, but, and even more pressing, to find a suitable new approach which would avoid the duality techniques. This is also done in this paper.

Before we move to the more detailed explanation of the new method, let us first mention one more interesting application of inequality~\eqref{GMU-inequality}, completely different in nature. Among many interesting function spaces that play significant role in functional analysis and its applications and that are normed with the help of some operator plugged inside an $L^p$-norm, one of the oldest and at the same time best known ones are those of Ces\`aro, and their natural companions, the spaces of Copson. Both have been extensively studied both in the discrete and the continuous variants. Let us be more precise. Our focus is on the continuous case.

Let $\mathcal{M}$ be the set of all measurable functions on $(0, \infty)$ and $\mathcal{M}^+$ be the collection of all nonnegative functions in $\mathcal{M}$. Let us denote by $\Ces_{p,q}(u,v)$ and $\Cop_{p,q}(u,v)$, the \textit{weighted Ces\`{a}ro and Copson function spaces}, defined, respectively, as the collection of all functions $f\in \mathcal{M}$ such that
\begin{equation*}
\|f\|_{\Ces_{p,q}(u,v)}= \bigg(\int_0^{\infty} \bigg(\int_0^t |f(s)|^p v(s)^p ds\bigg)^{\frac{q}{p}} u(t)^q dt \bigg)^{\frac{1}{q}}< \infty,
\end{equation*}
and
\begin{equation*}
\|f\|_{\Cop_{p,q}(u,v)}= \bigg(\int_0^{\infty} \bigg(\int_t^{\infty} |f(s)|^p v(s)^p ds\bigg)^{\frac{q}{p}} u(t)^q dt \bigg)^{\frac{1}{q}} < \infty,
\end{equation*}
for $0 < p,q \leq \infty$, with the usual modification when $p=\infty$ or $q=\infty$, where $u, v$ are \textit{weights}, that is, measurable, positive and finite a.e.~on $(0, \infty)$ functions. The name of the spaces is derived from the \textit{Ces\`aro operator} $f\mapsto \frac{1}{t}\int_{0}^{t}f(s)\,ds$ and the
\textit{Copson operator} $f\mapsto \int_{t}^{1}\frac{f(s)}{s}\,ds$, which define their important particular instances. These operators are  dual to each other
with respect to the $L^1$-pairing. The resulting spaces, along with their discrete companions, have been studied for many decades, for earliest occurrences see
e.g.~\cite{Ko-Kr-Le:48} or~\cite{Sh:70}. In order to keep this paper at reasonable length, we omit detailed information on the vast literature which is nowadays available on these spaces and restrict ourselves to referring an interested reader to the survey paper \cite{As-Ma:14} and to the references given therein.

In \cite{Be:96}, G. Bennett proved that the spaces $\Ces_{1,p}(t^{-1}, 1)$ and $\Cop_{1,p}(1,t^{-1})$ coincide when $p\in(1,\infty)$, and determined the space of multipliers between Ces\`{a}ro sequence spaces. Later, in \cite{Gr:98}, K.G. Grosse-Erdmann considered single-weighted Ces\`{a}ro and Copson sequence spaces and characterized the multipliers between them and $\ell^p$. At the same time he stated, however, that the multipliers between Ces\`{a}ro and Copson sequence spaces are difficult to treat. He also investigated integral analogues of these results. However, the description of pointwise multipliers between Ces\`{a}ro and Copson spaces remained open for both sequence and function spaces setting.

If $X$ and $Y$ are two (quasi-) Banach spaces of measurable functions on $(0, \infty)$, we say that $X$ is \textit{embedded} into $Y$, written $X \hookrightarrow Y$, if there exists a constant $c \in (0, \infty)$ such that $\|f\|_Y \leq c \|f\|_X$ for all $f \in X$. The \textit{optimal} (smallest possible) such $c$ is then $\|\Id\|_{X\rightarrow Y}$, where $\Id$ is the identity operator. A function $f$ on $(0,\infty)$ is called a \textit{pointwise multiplier} from $X$ to $Y$ if the pointwise product $fg$ belongs to $Y$ for each $g\in X$. The space of all such multipliers $M(X,Y)$ becomes a quasi-normed space, when endowed with the functional
\begin{equation*}
\|f\|_{M(X,Y)} = \sup_{g \nsim 0} \frac{\|fg\|_Y}{\|g\|_X}.
\end{equation*}
If $f$ is a weight, $\|\cdot\|_{Y}\colon\M\to[0,\infty]$ is a functional and $Y\subset\M$ given by
\begin{equation*}
    Y=\{f\in\M:\|f\|_{Y}<\infty\},
\end{equation*}
then we can define the \textit{weighted space} $Y_f = \{g\in\M\colon fg \in Y\}$ and $\|g\|_{Y_f} = \|fg\|_Y$. One clearly has
\begin{equation*}
\|f\|_{M(X,Y)} = \sup_{g \nsim 0} \frac{\|g\|_{Y_f}}{\|g\|_X}=\|\Id\|_{X\rightarrow Y_f}.
\end{equation*}
Therefore, the problem of finding pointwise multipliers between weighted Ces\`{a}ro and Copson function spaces reduces to that of characterizing embeddings between weighted Ces\`{a}ro and Copson function spaces. To characterize such embedding and to give a proper quantitative assessment of it means in fact to find the optimal constant $c$ in~\eqref{GMU-inequality}, providing it with another rather interesting array of applications.

Now let us move to the inequality~\eqref{GMU-inequality} itself. Apart from what has been said already, there exists significant motivation for considering this inequality. If $p_1=q_1$ or $p_2=q_2$, then it coincides with the notorious weighted Hardy inequality or the reverse Hardy inequality, respectively, whose significance in mathematics is beyond discussion. Plenty has been written on such inequalities, but still there is lot to be found about them. For the contemporary trends concerning Hardy-type inequalities one can check, for example, the relatively recent treatises \cite{Ku-Pe-Sa:17}, \cite{Gr:98} or \cite{Ev-Go-Op:08}. One of the earliest treatments of \eqref{GMU-inequality} goes back to \cite{Bo:70},
in which a characterization was given under conditions $p_1=p_2=1$, $q_1=q_2=p>1$, $v_1(t) = t^{-\beta-1}$, $v_2(t) = t^{\alpha-1}$, $u_1(t)^p = t^{\beta p-1}$ and $u_2(t)^p = t^{-\alpha p-1}$ for $t>0$ and for some $\alpha, \beta >0$. An important progress was achieved in \cite{Ca-Go-Ma-Pi:08}. In Theorem~2.3 of that paper, inequality \eqref{GMU-inequality} was characterized in the special case when $p_1 = p_2 = 1 $, without considering the pair of inner weights (that is, formally, for $v_1 = v_2 = 1$), and under the additional restriction $q \geq 1$. As already mentioned, the method was based on duality techniques that eventually reduced the problem to characterizing a reverse inequality for the Copson operator, which in turn was solvable by methods that had been developed earlier in~\cite{Si:03} and~\cite{Gr:98}, using also, for the particular case $p\le 1$, the rather useful inequalities (see~\cite{Ca-So:93}, \cite{Si-St:96} or~\cite{St:93})
\begin{equation*}
    \left(\int_{0}^{t}f^*(s)^{\frac{1}{p}}v(s)\,ds\right)^{p} \le p\int_{0}^tf^*(s)V(s)^{p-1}v(s)\,ds
\end{equation*}
and
\begin{equation*}
    \int_{0}^{\infty}f^*(t)V(t)^{\frac{1}{p}-1}v(t)\,dt \le C \left(\int_{0}^{\infty}f^*(t)^{p}v(s)\,dt\right)^{\frac{1}{p}}.
\end{equation*}

Using an approach similar to that from \cite{Ca-Go-Ma-Pi:08}, a characterization of the inequality \eqref{GMU-inequality} was obtained in \cite{Go-Mu-Un:17} for finite exponents, and in \cite{Un:21} for $p_1 = \infty$. Embeddings between weighted Ces\`{a}ro spaces are given in \cite{Un:20}. However, once again, the results are restricted to the case  $p_2 \leq q_2$,  for which the duality argument works. We should also note that duality method reduces the inequality \eqref{GMU-inequality} to an inequality involving iterated Hardy-type inequalities, which are also difficult to treat, although recently some useful techniques in that direction are being developed. In correspondence with the appearing of alternative methods, getting rid of various restrictions on parameters is recently becoming possible, see e.g.~\cite{Kr-Mi-Tu:21}. Using the characterizations of  the inequality \eqref{GMU-inequality}, pointwise multipliers between $\Cop_{p_1,q_1}(u_1,v_1)$ and $\Ces_{p_2,q_2}(u_2,v_2)$ are given in \cite{Go-Mu-Un:19}, but again only when $p_2 \leq q_2$.

The crucial idea in our new approach is to seek a more manageable version of~\eqref{GMU-inequality}. It is important to notice that the pool of competing functions in~\eqref{GMU-inequality} does not change when we replace $f$ by $(fv_1)^{p_1}$. Raising~\eqref{GMU-inequality} to $p_1$, introducing new parameters $p,q,r$ by
\begin{equation}\label{parameters}
    r=\frac{p_2}{p_1}, \,\, q= \frac{q_2}{p_1}, \,\, p= \frac{q_1}{p_1}
\end{equation}
and defining new weights $u,v,w$ by
\begin{equation}\label{weights}
    u=u_2^{q_2}, \,\, v= v_1^{-p_2} v_2^{p_2}, \,\, w= u_1^{q_1},
\end{equation}
we obtain a seemingly simpler, but in fact equivalent, form of~\eqref{GMU-inequality}, namely
\begin{equation}\label{main}
    \bigg(\int_0^{\infty} \bigg(\int_0^t f(s)^r v(s) ds \bigg)^{\frac{q}{r}} u(t) dt \bigg)^{\frac{1}{q}} \leq C \bigg(\int_0^{\infty} \bigg(\int_t^{\infty} f(s) ds \bigg)^{p} w(t) dt\bigg)^{\frac{1}{p}}.
\end{equation}
If $c$ and $C$ are the best constants respectively in~\eqref{GMU-inequality} and~\eqref{main}, then $c^{p_1} = C$.

So now it is the inequality~\eqref{main} that we have to worry about. We shall characterize it without any restrictions on parameters and weights. To achieve this aim, we will develop a new array of discretization and anti-discretization techniques which avoid duality. As follows from the analysis given above, the result will have some important consequences. First, one will be able to significantly extend the results of \cite{Ca-Go-Ma-Pi:08} by giving full characterization of all possible embeddings between spaces $S^q(w)$ and $\Lambda^p(v)$. Furthermore, we can remove restrictions from the characterization of pointwise multipliers between $\Cop_{p_1,q_1}(u_1,v_1)$ and $\Ces_{p_2,q_2}(u_2,v_2)$.

Apart from the motivation described above, there seems to be a rapidly growing interest in finding optimal bounds for inequalities of the form~\eqref{main}, both for nonnegative functions and for non-increasing functions, motivated by some very interesting connections to Poincar\'e or Sobolev inequalities, rearrangement estimates of BMO functions, the conjecture of Iwaniec concerning the norm of the Beurling operator, and more. The interested reader is kindly referred to checking the papers~\cite{Bo-So:11,Bo-So:19,Bo-So:20,Iw:82,Ko:14,Ko:19,Ko:20,Kr-Se:08,St:17,St:20}, and also the references given therein.

Let us briefly describe the structure of the paper. We state the main results in Section~\ref{S:main-results}. In Section~\ref{S:Emb-S-L},
we apply them to embeddings of the type $\Lambda^p(w) \hookrightarrow S^q(u)$ for $p,q \in (0, \infty)$. In Section~\ref{S:NotPre} we collect all auxiliary results and the main elements of the discretization technique. In Section~\ref{S:DisChar}, a discrete characterization of inequality \eqref{main} is obtained. Proofs of the main results are contained in Section~\ref{S:Proofs}.

The proofs of the main results are unavoidably quite heavy and very technical. This is natural - being deprived of the comfort of duality we have no other choice than to immerse into deep and extremely fine analysis of the discretized material, and finally to carefully antidiscretize it. Ever since the discretization techniques were first used in connection with classical Lorentz spaces \cite{Go-Pi:03}, much effort has been spent by many authors to obtain equally strong results by some less technical methods, but with almost no success.

We will write $A \lesssim B$ if there exists a constant $C\in (0, \infty)$ independent of appropriate quantities such that $A \leq C B$. We write $A \approx B$ if we have both $A \lesssim B $ and $B \lesssim A$.  We put $1/ (\pm \infty) =0$, $0/0 =0$, $0\cdot (\pm \infty)=0$. Moreover, we omit arguments of integrands as well as differentials in integrals when appropriate in order to keep the expository as short as possible.

\section{Main results}\label{S:main-results}

Note that inequality \eqref{main} holds only for trivial functions if $r>1$. Hence the assumption that $0 < r \leq 1$, which applies throughout the paper, constitutes no restriction (cf.~\cite[Lemma~4.1]{Go-Mu-Un:17}). For $a,b\in[0,\infty]$, $a<b$, we define
\begin{align}\label{Vr}
    V_r(a, b) =
    \begin{cases}
        \big(\int_a^b v^{\frac{1}{1-r}}\big)^{\frac{1-r}{r}} & \text{if $0<r<1$,}
            \\
        \esup\limits_{s \in (a, b)} v(s) & \text{if $r=1$.}
    \end{cases}
\end{align}

\begin{thm}\label{T:main}
Let $0 < r \leq 1$, $0 < p, q < \infty$ and let $u,v,w$ be weights on $(0,\infty)$.

\rm{(i)} If  $p \leq r \leq 1 \leq q$, then \eqref{main} holds for all $f\in\Mpl$ if and only if $C_1< \infty$, where
\begin{equation*}
    C_1 =  \esup_{x\in (0,\infty)} \bigg(\int_0^x w\bigg)^{-\frac{1}{p}} \bigg( \int_x^{\infty} u \bigg)^{\frac{1}{q}}  V_r(0,x).
\end{equation*}
Moreover, the best constant $C$ in \eqref{main} satisfies $C \approx C_1$.

\rm{(ii)} If  $p\leq q <1$, and $p \leq r\leq 1$, then \eqref{main} holds for all $f\in\Mpl$ if and only if $\max\{C_1, C_2\} < \infty$, where
\begin{equation*}
    C_2 =  \sup_{x\in (0,\infty)} \bigg(\int_0^x w\bigg)^{-\frac{1}{p}} \bigg( \int_0^x \bigg( \int_t^{\infty} u \bigg)^{\frac{q}{1-q}} u(t) V_r(0,t)^{\frac{q}{1-q}} dt \bigg)^{\frac{1-q}{q}}.
\end{equation*}
Moreover, the best constant $C$ in \eqref{main} satisfies $C \approx C_1 + C_2$.

\rm{(iii)} If  $r < p \leq q$, and $r \leq 1 \leq q$, then \eqref{main} holds for all $f\in\Mpl$ if and only if $\max\{C_1, C_3\} < \infty$, where
\begin{equation*}
    C_3 =  \sup_{x \in (0, \infty)} \bigg(\int_{x}^{\infty} u\bigg)^{\frac{1}{q}} \bigg(\int_0^{x} \bigg(\int_0^t w \bigg)^{-\frac{p}{p-r}} w(t) V_r(0,t)^{\frac{pr}{p-r}} dt \bigg)^{\frac{p-r}{pr}}.
\end{equation*}
Moreover, the best constant $C$ in \eqref{main} satisfies $C \approx C_1 + C_3$.

\rm{(iv)}  If  $r < p \leq q < 1$, then \eqref{main} holds for all $f\in\Mpl$ if and only if $\max\{C_1, C_2, C_3\} < \infty$. Moreover, the best constant $C$ in \eqref{main} satisfies $C \approx C_1 + C_2 + C_3$.

\rm{(v)} If $q < p \leq r \leq 1$, then \eqref{main} holds for all $f\in\Mpl$ if and only if $\max\{C_4, C_5\} < \infty$, where
\begin{equation*}
    C_4 = \bigg(\int_0^{\infty} \bigg(\int_t^{\infty} u\bigg)^{\frac{q}{p-q}} u(t)  \esup_{s\in (0, t)}  \bigg(\int_0^s w \bigg)^{-\frac{q}{p-q}} V_r(0,s)^{\frac{pq}{p-q}} dt \bigg)^{\frac{p-q}{pq}}
\end{equation*}
and
\begin{align*}
    C_5 &=  \bigg(\int_0^{\infty} \bigg(\int_0^x w\bigg)^{-\frac{p}{p-q}} w(x) \bigg( \int_0^{x} \bigg( \int_t^{x} u \bigg)^{\frac{q}{1-q}} u(t) V_r(0,t)^{\frac{q}{1-q}} dt \bigg)^{\frac{p(1-q)}{p-q}} dx \bigg)^{\frac{p-q}{pq}}
        \\
    & \hspace{1cm}+  \bigg(\int_0^{\infty}  w\bigg)^{-\frac{1}{p}} \bigg( \int_0^{\infty} \bigg( \int_t^{\infty} u \bigg)^{\frac{q}{1-q}} u(t) V_r(0,t)^{\frac{q}{1-q}} dt \bigg)^{\frac{1-q}{q}}.
\end{align*}
Moreover, the best constant $C$ in \eqref{main} satisfies $C \approx C_4 + C_5$.

\rm{(vi)} If  $q < p$, $q < 1$, $r < p$, and $r \leq 1$, then \eqref{main} holds for all $f\in\Mpl$ if and only if $\max\{C_5, C_6\} < \infty$, where
\begin{align*}
    C_6 &=  \bigg(\int_0^{\infty} \bigg(\int_0^x w\bigg)^{-2}w(x)  \sup_{y \in (0, x)} \bigg(\int_0^y w \bigg) \bigg(\int_y^x u(s)\bigg(\int_s^{\infty} u \bigg)^{\frac{q}{p-q}}ds\bigg) \times
        \\
    & \hspace{2cm}\times \bigg(\int_0^y \bigg(\int_0^s w \bigg)^{-\frac{p}{p-r}} w(s) V_r(0,s)^{\frac{pr}{p-r}} ds \bigg)^{\frac{q(p-r)}{r(p-q)}}  dx \bigg)^{\frac{p-q}{pq}}.
\end{align*}
Moreover, the best constant $C$ in \eqref{main} satisfies $C \approx C_5 + C_6$.

\rm{(vii)} If  $r \leq 1 \leq q < p$, then \eqref{main} holds for all $f\in\Mpl$ if and only if $\max\{C_6, C_7\} < \infty$, where\begin{align*}
C_7 &= \bigg(\int_0^{\infty} \bigg(\int_0^x w\bigg)^{-\frac{p}{p-q}} w(x) \esup_{t \in (0,x)} \bigg( \int_t^{\infty} u \bigg)^{\frac{p}{p-q}} V_r(0,t) ^{\frac{pq}{p-q}} dx \bigg)^{\frac{p-q}{pq}}
\notag\\
& \hspace{1cm}+ \bigg(\int_0^{\infty}  w\bigg)^{-\frac{1}{p}} \esup_{t \in (0, {\infty})} \bigg( \int_t^{\infty} u \bigg)^{\frac{1}{q}} V_r(0, t).
\end{align*}
Moreover, the best constant $C$ in \eqref{main} satisfies $C \approx C_6 + C_7$.
\end{thm}

If $ 0<r\leq q< p < 1$, an equivalent condition for the case (vi) of Theorem~\ref{T:main}, improved in a sense, is available. This result will be particularly handy for future applications, for instance to obtaining a characterization of multipliers between Ces\`aro and Copson spaces (for partial result in this direction see~\cite{Go-Mu-Un:19}).

\begin{thm}\label{T:equiv.solut.}
Let $0 < r \leq q <  p < 1 $ and let $u,v,w$ be weights on $(0, {\infty})$. Then \eqref{main} holds for all $f\in\Mpl$ if and only if $\max\{\mathcal{C}_5,  \mathcal{C}_6\} < \infty$, where
\begin{align*}
\mathcal{C}_5 &= \bigg(\int_0^{\infty} \bigg(\int_0^x w\bigg)^{-\frac{p}{p-q}} w(x) \bigg( \int_0^{x} \bigg( \int_t^{\infty} u \bigg)^{\frac{q}{1-q}} u(t) V_r(0,t)^{\frac{q}{1-q}} dt \bigg)^{\frac{p(1-q)}{p-q}} dx \bigg)^{\frac{p-q}{pq}}
\notag\\
& \hspace{1cm} + \bigg(\int_0^{\infty} w\bigg)^{-\frac{1}{p}} \bigg( \int_0^{\infty} \bigg( \int_t^{\infty} u \bigg)^{\frac{q}{1-q}} u(t) V_r(0, t)^{\frac{q}{1-q}}  dt \bigg)^{\frac{1-q}{q}}
\end{align*}
and
\begin{equation*}
\mathcal{C}_6 = \bigg(\int_0^{\infty} \bigg(\int_t^{\infty} u\bigg)^{\frac{q}{p-q}} u(t)  \bigg(\int_0^t \bigg(\int_0^s w \bigg)^{-\frac{p}{p-r}} w(s) V_r(0, s)^{\frac{pr}{p-r}} ds \bigg)^{\frac{q(p-r)}{r(p-q)}} dt \bigg)^{\frac{p-q}{pq}}.
\end{equation*}
Moreover, the best constant $C$ in \eqref{main} satisfies $C \approx \mathcal{C}_5 + \mathcal{C}_6$.
\end{thm}

\begin{rem}
The best constant in \eqref{GMU-inequality}, i.e.~the norm of the embedding $\Cop_{p_1,q_1}(u_1,v_1) \hookrightarrow \Ces_{p_2,q_2}(u_2,v_2)$, can be obtained from the characterization of \eqref{main} on changing the parameters as in \eqref{parameters} and \eqref{weights}. The norm of the converse embedding,  $\Ces_{p_1,q_1}(u_1,v_1) \hookrightarrow \Cop_{p_2,q_2}(u_2,v_2)$, can be reduced to \eqref{GMU-inequality} by using the change of variables $t\mapsto 1/t$.
\end{rem}

\section{Embeddings between $\Lambda^p(w)$ and $S^q(u)$} \label{S:Emb-S-L}

Constant functions always belong to $S^q(u)$ for every $q$ and $u$, because the functional $f\mapsto f^{**} - f^*$ vanishes on them, but they do not necessarily belong to $\Lambda^p(w)$ (that depends on $p$ and $w$). It is thus reasonable to study embeddings between $S^q(u)$ and $\Lambda^p(w)$ restricted to the collection
\begin{equation*}
\mathbb{A} = \{f\in \mathcal{M}^+ : \lim\limits_{t\rightarrow \infty} f^*(t) = 0\}.
\end{equation*}
In \cite[Proposition~7.2]{Ca-Go-Ma-Pi:08}, it was shown that if $0 < p,q < \infty$ and $u, w \in \mathcal{M}^+$, then the embedding $\Lambda^p(w) \hookrightarrow S^q(u)$, restricted to $\mathbb{A}$, holds if and only if
\begin{equation}\label{CGMP-inequality}
\bigg(\int_0^{\infty} \bigg(\frac{1}{t} \int_0^t h(s) ds\bigg)^q u(t) dt\bigg)^{\frac{1}{q}} \leq \mathcal{C} \bigg(\int_0^{\infty} \bigg(\int_t^{\infty} \frac{h(s)}{s}  ds\bigg)^p w(t) dt\bigg)^{\frac{1}{p}}
\end{equation}
is satisfied for every $h \in \mathcal{M}^+$. This result allows one to apply the results of the preceding section to the embedding problem.

\begin{thm} \label{app}
Let $0 < p < \infty$, $0 < q < 1$ and $u,w \in \mathcal{M}^+(0, \infty)$. Then the embedding
\begin{equation}\label{emb.SL}
\Lambda^p(w) \hookrightarrow S^q(u), \quad  \text{restricted to $\mathbb{A}$,}
\end{equation}
holds if and only if

\rm{(i)} either $p \leq q $ and $\max\{E_1, E_2\}<\infty$, where
\begin{equation*}
E_1= \sup_{x\in (0, \infty)} \bigg(\int_0^x w\bigg)^{-\frac{1}{p}} \sup_{t \in (0, x) } t \bigg( \int_t^{\infty} s^{-q} u(s) ds \bigg)^{\frac{1}{q}}
\end{equation*}
and
\begin{equation*}
E_2 =  \sup_{x\in (0, \infty)} \bigg(\int_0^x w\bigg)^{-\frac{1}{p}} \bigg( \int_0^x \bigg( \int_t^{\infty} s^{-q} u(s) ds \bigg)^{\frac{q}{1-q}} t^{\frac{q^2}{1-q}} u(t)  dt \bigg)^{\frac{1-q}{q}},
\end{equation*}

\rm{(ii)} or $q < p \leq 1$ and $\max\{E_3, E_4\}<\infty$, where
\begin{equation*}
E_3 = \bigg(\int_0^{\infty} \bigg(\int_t^{\infty} s^{-q} u(s) ds\bigg)^{\frac{q}{p-q}} t^{-q} u(t)  \sup_{s\in (0, t)}  \bigg(\int_0^s w \bigg)^{-\frac{q}{p-q}} s^{\frac{pq}{p-q}} dt \bigg)^{\frac{p-q}{pq}}
\end{equation*}
and
\begin{align*}
E_4 &=  \bigg(\int_0^{\infty} \bigg(\int_0^x w\bigg)^{-\frac{p}{p-q}} w(x) \bigg( \int_0^{x} \bigg( \int_t^{x} s^{-q} u(s) ds \bigg)^{\frac{q}{1-q}}  t^{\frac{q^2}{1-q}} u(t) dt \bigg)^{\frac{p(1-q)}{p-q}} dx \bigg)^{\frac{p-q}{pq}} \notag\\
& \hspace{1cm}+  \bigg(\int_0^{\infty}  w\bigg)^{-\frac{1}{p}} \bigg( \int_0^{\infty} \bigg( \int_t^{\infty} s^{-q} u(s) ds \bigg)^{\frac{q}{1-q}} t^{\frac{q^2}{1-q}} u(t) dt \bigg)^{\frac{1-q}{q}},
\end{align*}

\rm{(iii)} or $q < 1 < p$ and $\max\{E_4, E_5\} < \infty$, where
\begin{align*}
E_5 &=  \bigg(\int_0^{\infty} \bigg(\int_0^x w\bigg)^{-2}w(x)  \sup_{y \in (0, x)} \bigg(\int_0^y w \bigg) \bigg(\int_y^x s^{-q} u(s) \bigg(\int_s^{\infty} y^{-q} u(y) dy \bigg)^{\frac{q}{p-q}}ds\bigg)
\times  \\
& \hspace{2cm}\times \bigg(\int_0^y \bigg(\int_0^s w \bigg)^{-\frac{p}{p-1}} s^{\frac{p}{p-1}} w(s)  ds \bigg)^{\frac{q(p-1)}{p-q}}  dx \bigg)^{\frac{p-q}{pq}}.
\end{align*}
Moreover, if we denote by $C$ the optimal embedding constant in~\eqref{emb.SL}, then
\begin{equation*}
    C \approx
        \begin{cases}
            E_1+E_2 &\text{in the case \textup{(i)},}
                \\
            E_3 + E_4 &\text{in the case \textup{(ii)},}
                \\
            E_4 + E_5 &\text{in the case \textup{(iii)}.}
        \end{cases}
\end{equation*}
\end{thm}

\begin{proof}
It is easy to see that \eqref{CGMP-inequality} is equivalent to \eqref{main} when $r=1$ with $u(t)=t^{-q} u(t)$, $v(t)=t$, $t>0$. Therefore, the result follows on applying Theorem~\ref{T:main}, cases (ii), (v) and (vi) in order to establish the assertion in cases (i), (ii) and (iii), respectively.
\end{proof}

Let us note that Theorem~\ref{app} naturally complements~\cite[Corollary~7.3]{Ca-Go-Ma-Pi:08}, in which a characterization of~\eqref{emb.SL} is given for $1\le q<\infty$.

\begin{rem}
A converse embedding to~\eqref{emb.SL}, namely
\begin{equation*}
    S^p(u) \hookrightarrow \Lambda^q(w), \quad\text{restricted to $\mathbb{A}$,}
\end{equation*}
can be obtained from~\eqref{emb.SL} by a simple change of variables $t\mapsto \frac{1}{t}$ (cf.~\cite[Proposition~7.1]{Ca-Go-Ma-Pi:08}). One gets that~\eqref{emb.SL} holds if and only if
\begin{equation*}
    \Lambda^p(\tilde u ) \hookrightarrow S^q(\tilde w)
\end{equation*}
with $\tilde u(t) = u(1/t)t^{p-2}$ and $\tilde w(t) = w(1/t) t^{q-2}$, and with identical embedding constants.
\end{rem}

\section{Background material}\label{S:NotPre}

We start by quoting some basic definitions and facts concerning discretization and anti-discretization. We will denote by $\LHS(*)$ and $\RHS(*)$, the left hand side and right hand side of an inequality $(*)$, respectively.

\begin{defi}
Let $N, M \in \mathbb{Z} \cup \{-\infty, +\infty\}$, $N < M$ and $\{a_k\}_{k=N}^M$ be a sequence of positive real numbers. If

\begin{equation*}
    \sup\bigg\{ \frac{a_{k+1}}{a_k}: N < k < M\bigg\} < 1 \quad \text{or} \quad \inf\bigg\{ \frac{a_{k+1}}{a_k}: N < k < M\bigg\} > 1,
\end{equation*}
then we say that $\{a_k\}_{k=N}^M$ is \textit{strongly decreasing} or \textit{strongly increasing}, respectively.
\end{defi}

The following lemma will be frequently used throughout the paper. Its proof can be found in \cite[Proposition~2.1]{Go-He-St:96} for \eqref{dec.sum-sum} and \eqref{inc.sum-sum}  and \cite[Lemmas~3.2-3.4]{Go-Pi:03} for the remaining cases.

\begin{lem} \cite{Go-Pi:03,Go-He-St:96}
Let $\alpha\in (0,\infty)$ and $N, M \in  \mathbb{Z} \cup \{-\infty, +\infty\}$, $N < M$. Assume that $\{a_k\}_{k=N}^M$ and  $\{b_k\}_{k=N}^M$ are sequences of positive numbers.

\rm{(i)} If $\{a_k\}_{k=N}^M$ is non-increasing, then
\begin{equation} \label{dec.sup-sup}
    \sup_{N \leq k \leq M} a_k \sup_{N \leq i \leq k} b_i =  \sup_{N \leq k \leq M} a_k b_k.
\end{equation}
If, in addition, $\{a_k\}_{k=N}^M$ is strongly decreasing, then
\begin{equation}\label{dec.sum-sum}
    \sum_{k=N}^M a_k^\alpha \bigg(\sum_{i=N}^k b_i \bigg)^\alpha  \approx \sum_{k=N}^M a_k^\alpha b_k^\alpha,
\end{equation}
\begin{equation} \label{dec.sum-sup}
    \sum_{k=N}^M a_k \sup_{N \leq i\leq k} b_i   \approx \sum_{k=N}^M a_k b_k
\end{equation}
and
\begin{equation}\label{dec.sup-sum}
    \sup_{N \leq k\leq M} a_k \bigg(\sum_{i=N}^k b_i \bigg)^\alpha \approx \sup_{N \leq k\leq M} a_k b_k^\alpha.
\end{equation}

\rm{(ii)} If $\{a_k\}_{k=N}^M$ is strongly increasing, then
\begin{equation}\label{inc.sum-sum}
    \sum_{k=N}^M a_k^\alpha \bigg(\sum_{i=k}^M b_i \bigg)^\alpha  \approx \sum_{k=N}^M a_k^\alpha b_k^\alpha
\end{equation}
and
\begin{equation}\label{inc.sup-sum}
    \sup_{N \leq k\leq M} a_k \bigg(\sum_{i=k}^M b_i \bigg)^\alpha \approx \sup_{N \leq k\leq M} a_k b_k^\alpha.
\end{equation}
\end{lem}

The next lemma will be used in the proofs of our main results to handle the power of a sum. Its proof can be found in~\cite[Lemmas~1 and~1']{Be-Gr:05}.

\begin{lem} \rm{(power rules)}
Let $N, M \in \mathbb{Z}\cup \{-\infty, + \infty\}$, $N < M$ and $\beta >0$. Assume that $\{a_k\}_{k=N}^M$ is a sequence of non-negative real numbers such that $0 < \sum_{i=k}^M a_i <\infty$ for every $k \in \mathbb{Z}$ when  $0 < \beta < 1$. Then
\begin{equation}\label{power-rule for tails}
    \sum_{k=N}^{M} a_k \bigg(\sum_{j=k}^{M} a_j \bigg)^{\beta-1} \approx \bigg( \sum_{k=N}^{M} a_k \bigg)^{\beta}
\end{equation}
holds.
\end{lem}

The following lemma will be useful for finding various equivalent conditions for principal inequalities.

\begin{lem}
Let $s>0$, $N, M \in {\mathbb{Z}}\cup\{-\infty, +\infty\}$, $N<M$. Assume that $\{a_k\}_{k=N}^M$, $\{b_k\}_{k=N}^M$ are sequences of non-negative numbers such that $\{b_k\}_{k=N}^M$ is non-decreasing. If $N=-\infty$, denote by $b_N = \lim_{k\rightarrow -\infty} b_k$. Then,
\begin{equation}\label{difference-u}
	\sum_{k=N}^M a_k \bigg(\sum_{i=k}^M a_i \bigg)^s b_k \approx \sum_{k=N+1}^M (b_k  - b_{k-1} ) \bigg(\sum_{i= k}^M a_i\bigg)^{s+1} +   \bigg(\sum_{k=N}^M a_k\bigg)^{s+1} b_N.
\end{equation}

\end{lem}

\begin{proof}
Assume that $\{c_k\}_{k=N}^M$ is a sequence of non-negative numbers. Note that if $N= -\infty$, then $N+1$ is interpreted as $(-\infty)$, too. Observe that
$$
b_k = b_N + \sum_{i=N+1}^k(b_i - b_{i-1}).
$$
Then, changing the order of sums yields Abel's identity, more precisely
\begin{equation}\label{Abel}
\sum_{k=N}^M c_k b_k =   \sum_{k=N+1}^M (b_k  - b_{k-1}) \sum_{i= k}^M c_i + \bigg(\sum_{k=N}^M c_k\bigg) b_N.
\end{equation}

Set $c_k = a_k \big(\sum_{i=k}^M a_i\big)^s$ for $k\in[N,M)$ if $M=\infty$ or $k\in[N,M]$ if $M<\infty$. Applying \eqref{power-rule for tails}, we get
\begin{align*}
	\sum_{k=N}^M a_k \bigg(\sum_{i=k}^M a_i \bigg)^s b_k & =  \sum_{k=N+1}^M (b_k  - b_{k-1} ) \sum_{i= k}^M  a_i \big(\sum_{j=i}^M a_j\big)^s +  \bigg(\sum_{k=N}^M  a_k \big(\sum_{i=k}^M a_i\big)^s\bigg) b_N\\
	& \approx \sum_{k=N+1}^M (b_k  - b_{k-1} ) \bigg(\sum_{i= k}^M a_i\bigg)^{s+1} + \bigg(\sum_{k=N}^M a_k\bigg)^{s+1} b_N,
\end{align*}
establishing the claim.
\end{proof}

Let us further recall the discrete version of the classical Landau resonance theorem. Proofs can be found for example in \cite[Proposition~4.1]{Go-Pi:03}. The result can be also extracted from the more general embedding theorem in~\cite{Ka:81}.

\begin{thm}	\label{T:disc.lp.emb.}
Let $p,q \in (0,\infty)$ and $\{w_k\}_{k=N}^M$ and $\{v_k\}_{k=N}^M$ be two sequences of positive real numbers. Then inequality
\begin{equation}\label{Landau}
    \bigg(\sum_{k=N}^M a_k^q v_k^q \bigg)^{\frac{1}{q}} \leq
    C \bigg(\sum_{k=N}^M a_k^p w_k^p \bigg)^{\frac{1}{p}}
\end{equation}
holds for every sequence  $\{a_k\}_{k=N}^M$  of non-negative real numbers if and only if either

\rm(i) $p \leq q $ and
\begin{equation*}
    L_1= \sup_{N \leq k \leq M} v_k w_k^{-1} < \infty,
\end{equation*}
or

\rm(ii) $p > q$ and
\begin{equation*}
    L_2= \bigg(\sum_{k=N}^M v_k^{\frac{pq}{p-q}} w_k^{-\frac{pq}{p-q}} \bigg)^{\frac{p-q}{pq}} < \infty.
\end{equation*}
Moreover, if we denote by $C$ the optimal embedding constant in~\eqref{Landau}, then
\begin{equation*}
    C \approx
        \begin{cases}
            L_1 &\text{in the case \textup{(i)},}
                \\
            L_2 &\text{in the case \textup{(ii)}.}
        \end{cases}
\end{equation*}
\end{thm}

Next, we quote the characterization of the discrete Hardy inequality. In \cite[Theorem~1]{Be:91}, Bennett characterized the discrete Hardy inequality by breaking up into several cases depending on $p, q$. However, the condition presented there for the case $0 < q < p < 1$ had a complicated structure. Later, Grosse-Erdmann in \cite[Theorem~9.2]{Gr:98} presented an improved version of Bennett's theorem. The following theorem is a combination of these results.

\begin{thm}\label{T:disc.hardy}
Assume that $0 < p,q < \infty$ and  $\{a_k\}_{k=N}^M$ and
$\{b_k\}_{k=N}^M$ are sequences of non-negative real numbers. Then the inequality
\begin{equation}\label{E:disc.hardy}
    \bigg( \sum_{k=N}^M \bigg(
    \sum_{i =N}^k x_i b_i \bigg)^q a_k \bigg)^{\frac{1}{q}} \leq C \bigg(\sum_{k=N}^M x_k^p \bigg)^{\frac{1}{p}}
\end{equation}
holds for every sequence $\{x_k\}_{k=N}^M$ of non-negative real numbers if and only if
	
\rm(i) either $p \leq 1$, $p\leq q $ and
\begin{equation*}
    H_1= \sup_{N \leq k \leq M} \bigg(\sum_{i =k}^{M} a_i \bigg)^{\frac{1}{q}} b_k < \infty,
\end{equation*}

\rm(ii) or $q < p \le 1$ and
\begin{equation*}
    H_2= \bigg( \sum_{k=N}^M a_k \bigg(\sum_{i =k}^{M} a_i\bigg)^{\frac{q}{p-q}} \sup_{N \leq i\leq k} b_i^{\frac{qp}{p-q}} \bigg)^{\frac{p-q}{pq}} < \infty,
\end{equation*}

\rm(iii) or $1 < p $, $q < p$ and
\begin{equation*}
    H_3= \bigg( \sum_{k=N}^M a_k \bigg(\sum_{i = k}^M a_i\bigg)^{\frac{q}{p-q}} 	\bigg(\sum_{i= N}^{k} b_i^{\frac{p}{p-1}}\bigg)^{\frac{q(p-1)}{p-q}} \bigg)^{\frac{p-q}{pq}}< \infty,
\end{equation*}

\rm(iv) or $ 1< p \le q $ and
\begin{equation*}
    H_4=  \sup_{N\leq k \leq M} \bigg(\sum_{i = k}^{M} a_i\bigg)^{\frac{1}{q}} \bigg(\sum_{i= N}^{k} b_i^{\frac{p}{p-1}}\bigg)^{\frac{p-1}{p}} < \infty.
\end{equation*}
Moreover, if we denote by $C$ the best constant in~\eqref{E:disc.hardy}, then
\begin{equation*}
    C \approx
        \begin{cases}
            H_1 &\text{in the case \textup{(i)},}
                \\
            H_2 &\text{in the case \textup{(ii)},}
                \\
            H_3 &\text{in the case \textup{(iii)},}
                \\
            H_4 &\text{in the case \textup{(iv)}.}
        \end{cases}
\end{equation*}.
\end{thm}

Finally, we will need certain continuous analogues of some of the mentioned discrete results. We first recall that, for $0 < r \le 1$, $a,b\in[0,\infty]$, $a<b$, and $v$ a weight on $(a,b)$, one has
\begin{equation}\label{A(a,b)}
    V_r(a, b) = \sup_{f\in  \mathcal{M}^+(a, b)} \frac{\bigg(\int_{a}^{b} f^r v \bigg)^{\frac{1}{r}}} {\int_{a}^{b} f}.
\end{equation}

We round this section off with a modified version of the classical weighted Hardy inequality. For the proof, see e.g.~\cite{Ku-Pe-Sa:17} or~\cite{Si-St:96}.

\begin{thm}\label{T:cont.hardy}
Let $0 < r \leq 1 $, $0 < q < \infty$, $a,b\in[0,\infty]$, $a<b$, and assume that $u$ and $v$ are weight functions on $(a, b)$. Denote
\begin{equation}\label{B(a,b)}
    B(a, b) = \sup_{h\in  \mathcal{M}^+(a,b)} \frac{\bigg(\int_a^b \bigg(\int_a^t h^r v \bigg)^{\frac{q}{r}} u(t) dt \bigg)^{\frac{1}{q}}}{\int_a^b  h}.
\end{equation}

{\rm (i)} If $ 1 \leq q$, then
\begin{equation*}
    B(a,b) \approx  \esup_{t \in (a,b)} \bigg(\int_t^{b} u \bigg)^{\frac{1}{q}} V_r(a,t).
\end{equation*}
	
{\rm (ii)} If $ q < 1 $, then
\begin{equation*}
    B(a,b) \approx  \bigg( \int_a^b \bigg(\int_t^{b} u \bigg)^{\frac{q}{1-q}} u(t) V_r(a,t)^{\frac{q}{1-q}} dt \bigg)^{\frac{1-q}{q}}.
\end{equation*}
Moreover, in each case, the equivalence constants are independent of $a$ and $b$.
\end{thm}

\section{Discretization}\label{S:DisChar}

In this section, we formulate discrete characterizations of \eqref{main}. To this end, first we need to describe the fundamental elements of the discretization method.

\begin{defi}
Let $M \in \mathbb{Z}\cup \{+\infty\}$ and $\{x_k\}_{k=-\infty}^M$ be a strictly increasing sequence in $[a,b]$. We say that $\{x_k\}_{k=-\infty}^M$ is a covering sequence if $\lim_{k\rightarrow -\infty} x_k = 0 $ and, either $M = +\infty$ and $\lim_{k\rightarrow \infty} x_k = \infty$ or $M\in \mathbb{Z}$ and $x_M = \infty$.
\end{defi}
Denote by
$$
W(t) = \int_0^t w(s) ds, \quad t \in [0,\infty].
$$
\begin{defi}
Let $M \in \mathbb{Z}\cup \{+\infty\}$ and $w$ be a weight on $(0,\infty)$. We say that a covering sequence $\{x_k\}_{k=-\infty}^M \subset [0,\infty]$ is a discretizing sequence of $W$ if it satisfies $W(x_k) \approx 2^k$, $k \leq M$.
\end{defi}

We shall need the following result from~\cite{Go-Mi-Pi-Tu-Un:21}.

\begin{lem}\label{L:int-sup-equiv}
Let $\alpha \geq 0$. Assume that $w$ is a weight on $(0, \infty)$, $\{x_k\}_{k=-\infty}^{M}$ is a discretizing sequence of $W$. If $h$ is a non-negative non-increasing function on $(0, \infty)$, then
\begin{equation*}
    \int_0^{\infty} W(x)^{\alpha} w(x) h(x)  dx \approx \sum_{k=-\infty}^{M-1} 2^{k(\alpha +1)} h(x_k).
\end{equation*}
\end{lem}

Next we show that \eqref{main} is equivalent to two more manageable discrete inequalities.

\begin{lem}\label{L:equiv. ineq.}
Let $0 < r \leq 1$, $0 < p, q < \infty$ and let $u,v,w$ be weights on $(0,\infty)$. Then \eqref{main} holds for all $f\in\Mpl$ if and only if both
\begin{equation}\label{Ak inequality}
\bigg( \sum_{k=-\infty}^{M-1} \bigg( \sum_{i=-\infty}^{k} 2^{-i\frac{r}{p}} V_r(x_{i-1},x_i)^r a_i^r \bigg)^{\frac{q}{r}} \int_{x_k}^{x_{k+1}} u \bigg)^{\frac{1}{q}} \leq C' \bigg(\sum_{k=-\infty}^{M-1} a_k^p \bigg)^{\frac{1}{p}}
\end{equation}	
and
\begin{equation}\label{Bk inequality}
\bigg( \sum_{k=-\infty}^{M-1} 2^{-k\frac{q}{p}} B(x_k, x_{k+1})^q a_k^q \bigg)^{\frac{1}{q}} \leq C'' \bigg(\sum_{k=-\infty}^{M-1} a_k^p \bigg)^{\frac{1}{p}}
\end{equation}	
hold for every sequence of non-negative numbers $\{a_k\}_{k=-\infty}^{M-1}$, where $\{x_k\}_{k=-\infty}^{M}$ is the discretizing sequence of $W$, and $V_r(x_{k-1}, x_k)$ and $B(x_k, x_{k+1})$, $k \leq M-1$, are defined in \eqref{Vr} and \eqref{B(a,b)}, respectively. Moreover the best constants $C$, $C'$ and $C''$ in \eqref{main}, \eqref{Ak inequality} and \eqref{Bk inequality}, respectively, satisfy $C \approx C' + C''$.
\end{lem}

\begin{proof}
Let $\{x_k\}_{k=-\infty}^{M}$ be the discretizing sequence of $W$. By Lemma~\ref{L:int-sup-equiv} (with $\alpha = 0$) combined with \eqref{inc.sum-sum} for $\beta =p$, we get
\begin{equation}\label{Rhs-main}
\RHS{\eqref{main}} \approx \bigg(\sum_{k=-\infty}^{M-1} 2^k \bigg(\int_{x_k}^{x_{k+1}} f\bigg)^p\bigg)^{\frac{1}{p}}.
\end{equation}
On the other hand, we have
\begin{align*}
\LHS\eqref{main} &= \bigg(\sum_{k=-\infty}^{M-1} \int_{x_k}^{x_{k+1}} \bigg( \int_0^t f^r v \bigg)^{\frac{q}{r}} u(t) dt \bigg)^{\frac{1}{q}} \\
& \approx \bigg( \sum_{k=-\infty}^{M-1} \bigg( \int_0^{x_k} f^r v\bigg)^{\frac{q}{r}} \int_{x_k}^{x_{k+1}} u(t) dt \bigg)^{\frac{1}{q}} + \bigg(\sum_{k=-\infty}^{M-1} \int_{x_k}^{x_{k+1}} \bigg( \int_{x_k}^t f^r v \bigg)^{\frac{q}{r}} u(t) dt 	\bigg)^{\frac{1}{q}}.
\end{align*}
Therefore, \eqref{Rhs-main} shows that inequality \eqref{main} holds for all non-negative measurable functions $f$ on $(0,\infty)$ if and only if both
\begin{equation}\label{disc.main.1}
\bigg( \sum_{k=-\infty}^{M-1}  \bigg( \int_0^{x_k} f^r v \bigg)^{\frac{q}{r}} \int_{x_k}^{x_{k+1}} u(t) dt \bigg)^{\frac{1}{q}} \leq \mathfrak{C'} \bigg( \sum_{k=-\infty}^{M-1}  2^k \bigg( 	\int_{x_k}^{x_{k+1}} f \bigg)^p \bigg)^{\frac{1}{p}}
\end{equation}
and
\begin{equation}\label{disc.main.2}
\bigg( \sum_{k=-\infty}^{M-1} \int_{x_k}^{x_{k+1}} \bigg( \int_{x_k}^t f^r v \bigg)^{\frac{q}{r}} u(t) dt \bigg)^{\frac{1}{q}} \leq \mathfrak{C''} \bigg( \sum_{k=-\infty}^{M-1}  2^k \bigg( \int_{x_k}^{x_{k+1}} f \bigg)^p \bigg)^{\frac{1}{p}}
\end{equation}
hold for all non-negative measurable functions $f$ on $(0,\infty)$. Moreover, the best constants $C$, $\mathfrak{C'}$ and $\mathfrak{C''}$ in \eqref{main}, \eqref{disc.main.1} and \eqref{disc.main.2}, respectively, satisfy $C \approx \mathfrak{C'} + \mathfrak{C''}$.

We will now show that \eqref{disc.main.1} is equivalent to \eqref{Ak inequality}. Assume first that \eqref{disc.main.1} holds. By \eqref{A(a,b)}, there exist functions $h_k\in\Mpl$, $k \leq M-1$, on $(0,\infty)$ such that 	
\begin{equation*}
\supp h_k \subset [x_{k-1}, x_k], \quad \int_{x_{k-1}}^{x_{k}} h_k = 1 \quad \text{and} \quad \bigg( \int_{x_{k-1}}^{x_k} h_k^r v \bigg)^{\frac{1}{r}} \geq \frac{1}{2} V_r(x_{k-1}, x_k).
\end{equation*}
Fix a sequence of non-negative numbers $\{a_m\}_{m=-\infty}^{M-1}$ and set $f=\sum_{m=-\infty}^{M-1} 2^{-\frac{m}{p}} a_m h_m$. Then
\begin{align*}
\LHS\eqref{disc.main.1}  &=\bigg( \sum_{k =-\infty}^{M-1} \bigg( \sum_{i=-\infty}^k 2^{-i\frac{r}{p}} a_i^r \int_{x_{i-1}}^{x_i} h_i^r v \bigg)^{\frac{q}{r}} \int_{x_k}^{x_{k+1}} u \bigg)^{\frac{1}{q}} \\
& \gtrsim \bigg( \sum_{k =-\infty}^{M-1}  \bigg( \sum_{i=-\infty}^k  2^{-i\frac{r}{p}} a_i^r  V_r(x_{i-1}, x_i)^r \bigg)^{\frac{q}{r}} \int_{x_k}^{x_{k+1}} u \bigg)^{\frac{1}{q}}.
\end{align*}
On the other hand,
\begin{align*}
\RHS\eqref{disc.main.1} \approx \mathfrak{C'} \bigg( \sum_{k =-\infty}^{M-1} a_k^p \bigg( 	\int_{x_{k-1}}^{x_k} h_k \bigg)^p \bigg)^{\frac{1}{p}} = \mathfrak{C'} \bigg( \sum_{k =-\infty}^{M-1} a_k^p  \bigg)^{\frac{1}{p}}.
\end{align*}
Altogether, this establishes \eqref{Ak inequality} with $C' \lesssim \mathfrak{C'}$. Conversely, assume that \eqref{Ak inequality} holds. Since
\begin{equation*}
V_r(x_{k-1}, x_k) \geq \bigg( \int_{x_{k-1}}^{x_k} h^r v \bigg)^{\frac{1}{r}} \bigg(\int_{x_{k-1}}^{x_{k}} h 	\bigg)^{-1}, \quad k \leq M-1,
\end{equation*}
for all $h\in\Mpl$, \eqref{Ak inequality} yields
\begin{equation*}
\bigg( \sum_{k =-\infty}^{M-1} \bigg( \sum_{i=-\infty}^{k} 2^{-i\frac{r}{p}} \bigg( 	\int_{x_{i-1}}^{x_i} h^r v \bigg) \bigg(\int_{x_{i-1}}^{x_i} h \bigg)^{-r} a_i^r 	\bigg)^{\frac{q}{r}} \int_{x_k}^{x_{k+1}} u \bigg)^{\frac{1}{q}} \leq C' \bigg(\sum_{k =-\infty}^{M-1} a_k^p \bigg)^{\frac{1}{p}}.
\end{equation*}
On taking $a_k = 2^{\frac{k}{p}} \int_{x_{k-1}}^{x_k} h$, we get
\begin{equation*}
\bigg( \sum_{k =-\infty}^{M-1} \bigg( \sum_{i=-\infty}^{k} \int_{x_{i-1}}^{x_i} h^r v  \bigg)^{\frac{q}{r}} \int_{x_k}^{x_{k+1}} u \bigg)^{\frac{1}{q}} \leq C' \bigg(\sum_{k =-\infty}^{M-1} 2^k \bigg(\int_{x_{k-1}}^{x_k} h\bigg)^p \bigg)^{\frac{1}{p}}.
\end{equation*}
Therefore, \eqref{disc.main.1} holds with $\mathfrak{C'} \lesssim C'$. Hence, \eqref{disc.main.1} is equivalent to \eqref{Ak inequality} and $C' \approx \mathfrak{C'}$.

Now we will show that \eqref{disc.main.2} is equivalent to \eqref{Bk inequality}. Suppose that \eqref{disc.main.2} holds. By the definition of $B(x_k, x_{k+1})$, there exist functions $g_k\in\Mpl$, $k \leq M-1$, on $(0,\infty)$ such that
\begin{equation*}
\supp g_k \subset [x_k,x_{k+1}], \quad \int_{x_k}^{x_{k+1}} g_k =1 \quad \text{and} \quad \bigg( \int_{x_k}^{x_{k+1}} \bigg( \int_{x_k}^t g_k^r v\bigg)^\frac{q}{r} u(t) dt \bigg)^{\frac{1}{q}} \geq \frac{1}{2}  B(x_k, x_{k+1}).
\end{equation*}
Testing \eqref{disc.main.2} with $f= \sum_{m=-\infty}^{M-1} 2^{-\frac{m}{p}} a_m g_m$, where $\{a_m\}_{m=-\infty}^{M-1}$ is a sequence of non-negative numbers, we get \eqref{Bk inequality} with $C'' \lesssim \mathfrak{C''}$. Conversely, assume that \eqref{Bk inequality} holds. Since
\begin{equation*}
B(x_k, x_{k+1}) \geq  \bigg( \int_{x_k}^{x_{k+1}} \bigg( \int_{x_k}^t g^r v\bigg)^\frac{q}{r} u(t) dt \bigg)^{\frac{1}{q}} \bigg(\int_{x_k}^{x_{k+1}} g\bigg)^{-1}, \quad k \leq M-1,
\end{equation*}
for every $g\in\Mpl$, applying the latter inequality and putting $a_k = 2^{\frac{k}{p}} \int_{x_k}^{x_{k+1}} g$ in \eqref{Bk inequality}, we obtain \eqref{disc.main.2}. Moreover, $\mathfrak{C''} \lesssim C''$. Altogether, the claim follows.
\end{proof}

Now we are in position to state and prove the main result of this section, namely a discrete characterization of inequality \eqref{main}.

\begin{thm}\label{C:discrete solutions}
Let $0 < r \leq 1$, $0 < p, q < \infty$, let $u,v,w$ be weights on $(0,\infty)$ and let $\{x_k\}_{k=-\infty}^{M}$ be a discretizing sequence of $W$.

\rm{(i)} If  $p \leq r \leq 1 \leq q$, then \eqref{main} holds for all $f\in\Mpl$ if and only if $\max\{A_1, B_1\} < \infty$, where
\begin{equation*}
    A_1 = \sup_{k \leq M-1} 2^{-\frac{k}{p}} V_r(x_{k-1}, x_k) \bigg(\int_{x_k}^{\infty} u\bigg)^{\frac{1}{q}}
\end{equation*}
and
\begin{equation*}
    B_1 =  \sup_{k \leq M-1} 2^{-\frac{k}{p}} \sup_{t \in (x_k, x_{k+1}) } \bigg( \int_t^{x_{k+1}} u \bigg)^{\frac{1}{q}}  V_r(x_{k}, t).
\end{equation*}
Moreover, the best constant $C$ in \eqref{main} satisfies $C \approx A_1 + B_1$.

\rm{(ii)}  If  $p\leq q <1$, and $p \leq r\leq 1$, then \eqref{main} holds for all $f\in\Mpl$  if and only if $\max\{A_1, B_2\} < \infty$, where
\begin{equation*}
    B_2 =  \sup_{k\leq M-1} 2^{-\frac{k}{p}} \bigg( \int_{x_k}^{x_{k+1}} \bigg( \int_t^{x_{k+1}} u \bigg)^{\frac{q}{1-q}} u(t) V_r(x_k, t)^{\frac{q}{1-q}}  dt \bigg)^{\frac{1-q}{q}}.
\end{equation*}
Moreover, the best constant $C$ in \eqref{main} satisfies $C \approx A_1 + B_2$.
	
\rm{(iii)} If $r < p \leq q$, and $r \leq 1 \leq q$, then \eqref{main} holds for all $f\in\Mpl$  if and only if $\max\{A_2, B_1\} < \infty$, where
\begin{equation*}
    A_2 =  \sup_{k\leq M-1} \bigg(\int_{x_k}^{\infty} u\bigg)^{\frac{1}{q}} \bigg(\sum_{i=-\infty}^k 2^{-i\frac{r}{p-r}} V_r(x_{i-1}, x_i)^{\frac{pr}{p-r}} \bigg)^{\frac{p-r}{pr}}.
\end{equation*}
Moreover, the best constant $C$ in \eqref{main} satisfies $C \approx A_2 + B_1$.

\rm{(iv)} If $r < p \leq q < 1$, then \eqref{main} holds for all $f\in\Mpl$  if and only if $\max\{A_2, B_2\}< \infty$. Moreover, the best constant $C$ in \eqref{main} satisfies $C \approx A_2 + B_2$.

\rm{(v)} If  $q < p \leq r \leq 1$, then \eqref{main} holds for all $f\in\Mpl$  if and only if $\max\{A_3, B_3\} < \infty$, where
\begin{equation*}
    A_3 = \bigg(\sum_{k=-\infty}^{M-1} \bigg(\int_{x_k}^{x_{k+1}} u\bigg) \bigg(\int_{x_k}^{\infty} u\bigg)^{\frac{q}{p-q}} \sup_{i\leq k}  2^{-i\frac{q}{p-q}} V_r(x_{i-1},x_i)^{\frac{pq}{p-q}} \bigg)^{\frac{p-q}{pq}}
\end{equation*}
and
\begin{equation*}
    B_3 =   \bigg(\sum_{k=-\infty}^{M-1} 2^{-k\frac{q}{p-q}} \bigg( \int_{x_k}^{x_{k+1}} \bigg( \int_t^{x_{k+1}} u \bigg)^{\frac{q}{1-q}} u(t) V_r(x_k,t)^{\frac{q}{1-q}} dt \bigg)^{\frac{p(1-q)}{p-q}} \bigg)^{\frac{p-q}{pq}}.
\end{equation*}
Moreover, the best constant $C$ in \eqref{main} satisfies $C \approx A_3 + B_3$.
	
\rm{(vi)} If  $q < p$, $q < 1$, $r < p$ and $r \leq 1$, then \eqref{main} holds for all $f\in\Mpl$  if and only if $\max\{A_4, B_3\} < \infty$, where
\begin{equation*}
    A_4 =  \bigg(\sum_{k=-\infty}^{M-1} \bigg(\int_{x_k}^{x_{k+1}} u \bigg) \bigg(\int_{x_k}^{\infty} u \bigg)^{\frac{q}{p-q}} \bigg( \sum_{i=-\infty}^{k} 2^{-i\frac{r}{p-r}} V_r(x_{i-1}, x_i)^{\frac{pr}{p-r}} \bigg)^{\frac{q(p-r)}{r(p-q)}}  \bigg)^{\frac{p-q}{pq}}.
\end{equation*}
Moreover, the best constant $C$ in \eqref{main} satisfies $C \approx A_4 + B_3$.
	
\rm{(vii)} If  $r \leq 1 \leq q < p$, then \eqref{main} holds for all $f\in\Mpl$ if and only if $\max\{A_4, B_4\} < \infty$, where
\begin{equation*}
    B_4 =  \bigg( \sum_{k=-\infty}^{M-1} 2^{-k\frac{q}{p-q}} \sup_{t \in (x_k, x_{k+1})}
    \bigg(\int_t^{x_{k+1}} u \bigg)^{\frac{p}{p-q}} V_r(x_k,t)^{\frac{pq}{p-q}} \bigg)^{\frac{p-q}{pq}}.
\end{equation*}
Moreover, the best constant $C$ in \eqref{main} satisfies $C \approx A_4 + B_4$.
\end{thm}

\begin{proof}
By Lemma~\ref{L:equiv. ineq.}, the best constant in \eqref{main} satisfies $C \approx C' + C''$, where $C'$ and $C''$ are the best constant in \eqref{Ak inequality} and \eqref{Bk inequality}, respectively. Characterization of $C'$ can be obtained by Theorem~\ref{T:disc.hardy} combined with~\eqref{A(a,b)}, and he characterization of $C''$ follows from Theorems~\ref{T:disc.lp.emb.} and~\ref{T:cont.hardy}.
\end{proof}

\begin{rem}
Given a discretizing sequence $\{x_k\}_{k=-\infty}^{M}$ of $W$, we can write
\begin{align}\label{Vr(0,xk)}
V_r(0, x_k) =
    \begin{dcases}
        \bigg(\sum_{i=-\infty}^k \int_{x_{i-1}}^{x_i} v^{\frac{1}{1-r}}\bigg)^{\frac{1-r}{r}}, & \text{if $0<r<1$,}
            \\
        \sup_{i\leq k} \esup\limits_{s \in (x_{i-1}, x_i)} v(s), & \text{if $r=1$}.
    \end{dcases}
\end{align}
Since $\{2^{-\frac{k}{p}} \big( \int_{x_k}^{\infty} u \big)^{\frac{1}{q}}\}$ is strongly decreasing, we can infer from \eqref{Vr(0,xk)} on applying either \eqref{dec.sup-sum} (when $r < 1$) or \eqref{dec.sup-sup} (when $r=1$) that
\begin{equation}\label{A1-equiv}
\sup_{k \leq M-1} 2^{-\frac{k}{p}} \bigg( \int_{x_k}^{\infty} u \bigg)^{\frac{1}{q}} V_r(0,x_k) \approx \sup_{k \leq M-1} 2^{-\frac{k}{p}} \bigg( \int_{x_k}^{\infty} u \bigg)^{\frac{1}{q}} V_r(x_{k-1},x_k) = A_1.
\end{equation}
Consequently, if $A_1 < \infty$, then, owing to
$\lim_{k\rightarrow -\infty} 2^{-\frac{k}{p}} \big( \int_{x_k}^{\infty} u \big)^{\frac{1}{q}} = \infty$,
we have
\begin{equation}\label{lim-0}
\lim_{k\rightarrow -\infty} V_r(0, x_k) = 0.
\end{equation}
Moreover, it is evident that $A_1 \leq A_2$. Additionally, using \eqref{power-rule for tails} with $\beta = \frac{p}{p-q}$, we obtain for each $k \in \mathbb{Z}$, $k \leq M-1$, that
\begin{align}\label{A1<A3}
A_1 &  = \sup_{k \leq M-1} 2^{-\frac{k}{p}} V_r(x_{k-1}, x_k) \bigg(\sum_{i=k}^{M-1} \int_{x_i}^{x_{i+1}} u \bigg)^{\frac{1}{q}}\notag \\
& \approx \sup_{k \leq M-1} 2^{-\frac{k}{p}} V_r(x_{k-1}, x_k) \bigg(\sum_{i=k}^{M-1} \bigg(\int_{x_i}^{x_{i+1}} u\bigg) \bigg(\int_{x_i}^{\infty} u\bigg)^{\frac{q}{p-q}}  \bigg)^{\frac{p-q}{pq}} \notag\\
& \leq  \sup_{k \leq M-1} \bigg(\sum_{i=k}^{M-1} \bigg(\int_{x_i}^{x_{i+1}} u\bigg) \bigg(\int_{x_i}^{\infty} u\bigg)^{\frac{q}{p-q}} \sup_{m \leq i} 2^{-m\frac{q}{p-q}} V_r(x_{m-1}, x_m)^{\frac{pq}{p-q}} \bigg)^{\frac{p-q}{pq}} \notag\\
& = A_3.
\end{align}
We can similarly show that $A_1 \lesssim A_4$. Therefore, the conditions $A_2<\infty$, $A_3 < \infty$ and $A_4 < \infty$ yield \eqref{lim-0} as well. Consequently, without loss of generality we can (and will) assume that
\begin{equation*}
    \lim_{t\rightarrow 0+} V_r(0,t) = 0.
\end{equation*}
\end{rem}

\section{Proofs of the main results} \label{S:Proofs}

\begin{proof}[Proof of Theorem \ref{T:main}]
\rm{(i)} We  have by Theorem~\ref{C:discrete solutions}(i) that the best constant $C$ in \eqref{main} satisfies $C \approx A_1 + B_1$. We will prove that $C_1 \approx A_1 + B_1$. Since  $\{ x_k\}_{k=-\infty}^{M}$ is the discretizing sequence of $W$, we have
\begin{align}\label{C1-equiv}
C_1 &= \sup_{k \leq M-1} \esup_{t \in (x_k, x_{k+1})} \bigg(\int_0^t w\bigg)^{-\frac{1}{p}}\bigg( \int_t^{\infty} u \bigg)^{\frac{1}{q}} V_r(0,t) \notag \\
&\approx\sup_{k \leq M-1} 2^{-\frac{k}{p}} \esup_{t \in (x_k, x_{k+1})} \bigg( \int_t^{\infty} u \bigg)^{\frac{1}{q}} V_r(0,t).
\end{align}
Using the fact that
\begin{equation}\label{V-cut}
V_r(0,t) \approx V_r(0, x_k) + V_r(x_k, t) \quad\text{for $t \in (x_k, x_{k+1})$,}
\end{equation}
we obtain that
\begin{align*}
C_1 & \approx  \sup_{k \leq M-1} 2^{-\frac{k}{p}} \bigg( \int_{x_k}^{\infty} u \bigg)^{\frac{1}{q}} V_r(0,x_k) +  \sup_{k \leq M-1} 2^{-\frac{k}{p}} \esup_{t \in (x_k, x_{k+1})} \bigg( \int_t^{\infty} u \bigg)^{\frac{1}{q}} V_r(x_k,t).
\end{align*}
 Then, in view of \eqref{A1-equiv}, we arrive at
\begin{align*}
C_1 \approx  A_1 + B_1 +  \sup_{k \leq M-2} 2^{-\frac{k}{p}} \bigg( \int_{x_{k+1}}^{\infty} u \bigg)^{\frac{1}{q}} V_r(x_k, x_{k+1}) \approx A_1 + B_1.
\end{align*}
\medskip

\rm{(ii)} Similarly as above, owing to Theorem~\ref{C:discrete solutions}(ii), it suffices to show that that $C_1 + C_2 \approx A_1 + B_2$. Observe that, for any $x, y \in \mathbb{R}$,
\begin{align}
\esup_{t \in (x, y)} \bigg(\int_t^{y} u \bigg)^{\frac{1}{q}} V_r(x, t) & \approx \esup_{t \in (x, y)} \bigg(\int_t^y \bigg(\int_s^y u \bigg)^{\frac{q}{1-q}} u(s) ds \bigg)^{\frac{1-q}{q}} V_r(x, t) \notag\\
&  \leq  \bigg(\int_x^y \bigg(\int_s^y u \bigg)^{\frac{q}{1-q}} u(s) V_r(x, s)^{\frac{q}{1-q}} ds \bigg)^{\frac{1-q}{q}}. \label{upper-B1-B2}
\end{align}
We conclude that $B_1 \lesssim B_2$. In addition, we proved in the preceding case that $C_1 \approx A_1 + B_1$. Thus, $A_1 \lesssim C_1 \lesssim A_1 + B_2$.
On the other hand, since $\{ x_k\}_{k=-\infty}^{M}$ is the discretizing sequence of $W$, properties of $\{ x_k\}_{k=-\infty}^{M}$ and \eqref{dec.sup-sum} yield that
\begin{align*}
C_2  & = \sup_{k \leq M-1} \sup_{x\in (x_k, x_{k+1})} \bigg(\int_0^x w\bigg)^{-\frac{1}{p}} \bigg( \int_0^x \bigg( \int_t^{\infty} u \bigg)^{\frac{q}{1-q}} u(t) V_r(0,t)^{\frac{q}{1-q}} dt \bigg)^{\frac{1-q}{q}} \\
&\approx \sup_{k \leq M-1} 2^{-\frac{k}{p}} \bigg( \sum_{i=-\infty}^k  \int_{x_i}^{x_{i+1}} \bigg( \int_t^{\infty} u \bigg)^{\frac{q}{1-q}} u(t) V_r(0,t)^{\frac{q}{1-q}} dt \bigg)^{\frac{1-q}{q}} \\
&\approx \sup_{k \leq M-1} 2^{-\frac{k}{p}} \bigg(  \int_{x_k}^{x_{k+1}} \bigg( \int_t^{\infty} u \bigg)^{\frac{q}{1-q}} u(t) V_r(0,t)^{\frac{q}{1-q}} dt \bigg)^{\frac{1-q}{q}}.
\end{align*}
It is thus clear that $B_2 \lesssim C_2$, and, in turn, $A_1 + B_2 \lesssim C_1 + C_2$. It remains to show that $C_2 \lesssim A_1 + B_2$. Using \eqref{V-cut}, we have
\begin{align*}
C_2  &\approx  \sup_{k \leq M-1} 2^{-\frac{k}{p}} \bigg( \int_{x_k}^{x_{k+1}} \bigg( \int_t^{\infty} u \bigg)^{\frac{q}{1-q}} u(t) V_r(x_k,t)^{\frac{q}{1-q}} dt \bigg)^{\frac{1-q}{q}} \\
& \quad + \sup_{k \leq M-1} 2^{-\frac{k}{p}} V_r(0,x_k) \bigg( \int_{x_k}^{x_{k+1}} \bigg( \int_t^{\infty} u \bigg)^{\frac{q}{1-q}} u(t)  dt
\bigg)^{\frac{1-q}{q}}.
\end{align*}
Integration by parts gives
\begin{align*}
C_2  &\lesssim  \sup_{k \leq M-1} 2^{-\frac{k}{p}} \bigg( \int_{x_k}^{x_{k+1}} \bigg( \int_t^{\infty} u \bigg)^{\frac{1}{1-q}} d \big[V_r(x_k,t)^{\frac{q}{1-q}}\big] \bigg)^{\frac{1-q}{q}} \\
&\quad + \sup_{k \leq M-1} 2^{-\frac{k}{p}} \lim_{t \rightarrow x_k+} \bigg( \int_t^{\infty} u \bigg)^{\frac{1}{q}} V_r(x_k,t)\\
& \quad + \sup_{k \leq M-1} 2^{-\frac{k}{p}} V_r(0,x_k) \bigg( \int_{x_k}^{\infty} u \bigg)^{\frac{1}{q}}\\
& = C_{2,1} + C_{2,2} + C_{2,3}.
\end{align*}
Using \eqref{A1-equiv}, we get $C_{2,3} \approx A_1$. On the other hand, note that
\begin{align} \label{limit<supremum}
\lim_{t \rightarrow x_k+} \bigg(\int_t^{\infty} u \bigg)^{\frac{1}{q}} V_r(x_k, t) \leq \esup_{s \in (x_k, x_{k+1})} \bigg(\int_s^{\infty} u \bigg)^{\frac{1}{q}} V_r(x_k, s).
\end{align}
Thus, \eqref{limit<supremum} gives
\begin{align*}
C_{2,2} & \lesssim \sup_{k \leq M-1} 2^{-\frac{k}{p}} \esup_{t \in (x_k, x_{k+1})} \bigg(\int_t^{x_{k+1}} u \bigg)^{\frac{1}{q}} V_r(x_k, t)
+ \sup_{k \leq M-2} 2^{-\frac{k}{p}} \bigg(\int_{x_{k+1}}^{\infty} u \bigg)^{\frac{1}{q}} V_r(x_k, {x_{k+1}}).
\end{align*}
Using \eqref{upper-B1-B2}, we obtain $C_{2,2} \lesssim B_2 + A_1$. Moreover,
\begin{align*}
C_{2,1} & \approx  \sup_{k \leq M-1} 2^{-\frac{k}{p}} \bigg( \int_{x_k}^{x_{k+1}} \bigg( \int_t^{x_{k+1}} u \bigg)^{\frac{1}{1-q}} d \big[V_r(x_k,t)^{\frac{q}{1-q}}\big] \bigg)^{\frac{1-q}{q}} \\
& \quad + \sup_{k \leq M-2} 2^{-\frac{k}{p}} \bigg( \int_{x_{k+1}}^{\infty} u \bigg)^{\frac{1}{q}} \bigg( \int_{x_k}^{x_{k+1}}  d \big[V_r(x_k,t)^{\frac{q}{1-q}}\big]\bigg)^{\frac{1-q}{q}}.
\end{align*}
Thus, integration by parts yields that $ C_{2,1}  \lesssim B_2 + A_1$. Altogether, we arrive at $C_2 \lesssim C_{2,1} + C_{2,2} + C_{2,3} \lesssim A_1 + B_2$.

\medskip

\rm{(iii)} By Theorem~\ref{C:discrete solutions}(iii), we have that the best constant $C$ in \eqref{main} satisfies $C \approx A_2 + B_1$. We will now show that $C_1+C_3 \approx A_2 + B_1$. We have shown in case (i) that $C_1 \approx A_1 + B_1$. Since $A_1 \leq A_2$ is obvious, we get $B_1 \lesssim C_1 \lesssim A_2 + B_1$. Next, we will find a suitable upper estimate for $C_3$. As $\{x_k\}_{k=-\infty}^M$ is a discretizing sequence of $W$, we have
\begin{align*}
C_3 & = \sup_{k \leq M-1} \sup_{x \in (x_k, x_{k+1})} \bigg(\int_{x}^{\infty} u\bigg)^{\frac{1}{q}} \bigg(\int_0^{x} \bigg(\int_0^t w \bigg)^{-\frac{p}{p-r}} w(t) V_r(0,t)^{\frac{pr}{p-r}} dt \bigg)^{\frac{p-r}{pr}} \\
& \approx \sup_{k \leq M-1}  \bigg(\int_{x_k}^{\infty} u\bigg)^{\frac{1}{q}} \bigg(\int_0^{x_{k}} \bigg(\int_0^t w \bigg)^{-\frac{p}{p-r}} w(t) V_r(0,t)^{\frac{pr}{p-r}} dt \bigg)^{\frac{p-r}{pr}} \\
    &\quad +  \sup_{k \leq M-1} \sup_{x \in (x_k, x_{k+1})} \bigg(\int_{x}^{\infty} u\bigg)^{\frac{1}{q}} \bigg(\int_{x_k}^x \bigg(\int_0^t w \bigg)^{-\frac{p}{p-r}} w(t) V_r(0,t)^{\frac{pr}{p-r}} dt \bigg)^{\frac{p-r}{pr}}\\
&= C_{3,1} + C_{3,2}.
\end{align*}
On the other hand, in view of  \eqref{Vr(0,xk)}, using \eqref{dec.sum-sum} when $r<1$ and  \eqref{dec.sum-sup} when $r=1$,
\begin{equation}\label{A2-equiv}
\sum_{i= -\infty}^k 2^{-i\frac{r}{p-r}} V_r(0, x_i)^{\frac{pr}{p-r}} \approx \sum_{i= -\infty}^k 2^{-i\frac{r}{p-r}} V_r(x_{i-1}, x_i)^{\frac{pr}{p-r}}
\end{equation}
holds. Applying \eqref{A2-equiv}, we obtain
\begin{align}\label{lower-A2}
    \int_0^{x_{k}} \bigg(\int_0^t w \bigg)^{-\frac{p}{p-r}} w(t) V_r(0,t)^{\frac{pr}{p-r}} dt & = \sum_{i= -\infty}^k  \int_{x_{i-1}}^{x_i} \bigg(\int_0^t w\bigg)^{-\frac{p}{p-r}} w(t) V_r(0,t)^{\frac{pr}{p-r}} dt  \notag
        \\
    & \lesssim \sum_{i= -\infty}^k 2^{-i\frac{r}{p-r}} V_r(x_{i-1}, x_i)^{\frac{pr}{p-r}}.
\end{align}
Thus, \eqref{lower-A2} implies $C_{3,1} \lesssim A_2$. Moreover, \eqref{C1-equiv} together with $C_1 \lesssim A_2 +B_1$ yield
\begin{align*}
    C_{3,2} &\leq \sup_{k \leq M-1} 2^{-\frac{k}{p}}  \esup_{x \in (x_k, x_{k+1})} \bigg(\int_{x}^{\infty} u\bigg)^{\frac{1}{q}} V_r(0,x) \approx C_1 \lesssim A_2 + B_1.
\end{align*}
Consequently, we arrive at $C_1 + C_3 \lesssim A_2 + B_1$. We will be done once we show that $A_2 \lesssim C_3 + C_1$. To this end, we need the estimate
\begin{align} \label{upper-A2}
    \sum_{i= -\infty}^{k-1}  2^{-i\frac{r}{p-r}} V_r(x_{i-1}, x_i)^{\frac{pr}{p-r}}
    	&\approx \sum_{i= -\infty}^{k-1}  \int_{x_i}^{x_{i+1}} \bigg(\int_0^t w \bigg)^{-\frac{p}{p-r}}  w(t) dt V_r(x_{i-1}, x_i)^{\frac{pr}{p-r}} \notag
            \\
    	& \le \int_{0}^{x_{k}} \bigg(\int_0^t w \bigg)^{-\frac{p}{p-r}}  w(t) V_r(0, t)^{\frac{pr}{p-r}} dt.
\end{align}
In view of \eqref{upper-A2}, we have
\begin{align*}
    A_2 & \approx  A_1 + \sup_{k\leq M-1} \bigg(\int_{x_k}^{\infty} u\bigg)^{\frac{1}{q}} \bigg( \int_{0}^{x_{k}} \bigg(\int_0^t w \bigg)^{-\frac{p}{p-r}}  w(t) V_r(0, t)^{\frac{pr}{p-r}} dt \bigg)^{\frac{p-r}{pr}}.
\end{align*}
Since $A_1 \lesssim C_1$, we arrive at
\begin{align*}
A_2 & \lesssim  C_1 +  \sup_{k\leq M-1}  \sup_{x \in (x_k, x_{k+1})} \bigg(\int_{x}^{\infty} u\bigg)^{\frac{1}{q}} \bigg( \int_{0}^{x} \bigg(\int_0^t w \bigg)^{-\frac{p}{p-r}}  w(t) V_r(0, t)^{\frac{pr}{p-r}} dt \bigg)^{\frac{p-r}{pr}}\\
& = C_1 + C_3,
\end{align*}
and the assertion follows.
\medskip

\rm{(iv)} By Theorem~\ref{C:discrete solutions}(iv), the best constant $C$ in \eqref{main} satisfies $C \approx A_2 + B_2$. We will show that $C_1 + C_2 + C_3 \approx A_2 + B_2$. Note that, in the proofs of the cases (i)-(iii), the actual positions of the parameters did not play any role. We can survey these results as
\begin{align*}
    & C_1 \approx A_1+B_1 \leq A_2 + B_2,
        \\
    & C_2 \lesssim A_1 + B_2 \leq A_2 + B_2,
        \\
    &C_3 \leq A_2 + B_1 \lesssim A_2 + B_2.
\end{align*}
On the other hand, we have $A_2 \lesssim C_3 + C_1$ and $B_2 \lesssim C_2$. Combining these estimates gives $A_2 +B_2 \lesssim C_1 + C_2 + C_3 \lesssim A_2 + B_2$, and the assertion follows once again.
\medskip

\rm{(v)} By Theorem~\ref{C:discrete solutions}(v) we have $C \approx A_3+B_3$. We will prove that $A_3+B_3 \approx  C_4 + C_5$. Taking \eqref{Vr(0,xk)} into consideration, we observe that, using \eqref{dec.sup-sum} when $r < 1$ and  \eqref{dec.sup-sup} when $r=1$,
\begin{equation}\label{sup-dec}
\sup_{i\leq {k}} 2^{-i\frac{q}{p-q}} V_r(0,x_i)^{\frac{pq}{p-q}} \approx  \sup_{i\leq {k}} 2^{-i\frac{q}{p-q}} V_r(x_{i-1},x_i)^{\frac{pq}{p-q}}.
\end{equation}
By \eqref{sup-dec}, we get the chain of relations
\begin{align} \label{A3-equiv}
\esup_{t \in (0, x_k)} \bigg(\int_0^t w\bigg)^{-\frac{q}{p-q}} V_r(0,t)^{\frac{pq}{p-q}}  & = \sup_{i\leq {k}} \esup_{t \in (x_{i-1},x_i)} \bigg(\int_0^t w\bigg)^{-\frac{q}{p-q}} V_r(0,t)^{\frac{pq}{p-q}} \notag\\
& \approx  \sup_{i\leq {k}} 2^{-i\frac{q}{p-q}} V_r(x_{i-1},x_i)^{\frac{pq}{p-q}}.
\end{align}
We will next show that $C_4 \lesssim A_3+B_3$. Integration by parts gives
\begin{align*}
C_4^{\frac{pq}{p-q}} & \approx \sum_{k= -\infty}^{M-2} \int_{x_k}^{x_{k+1}} \bigg( \int_t^{\infty} u\bigg)^{\frac{q}{p-q}} u(t)  \esup_{s\in (0, t)}  \bigg( \int_0^s w \bigg)^{-\frac{q}{p-q}} V_r(0,s)^{\frac{pq}{p-q}} dt \\
    & \quad + \int_{x_{M-1}}^{\infty} \bigg( \int_t^{\infty} u\bigg)^{\frac{q}{p-q}} u(t)  \esup_{s\in (0, t)}  \bigg( \int_0^s w \bigg)^{-\frac{q}{p-q}} V_r(0,s)^{\frac{pq}{p-q}} dt\\
& \lesssim \sum_{k= -\infty}^{M-2} \bigg[ \bigg(\int_{x_k}^{\infty} u\bigg)^{\frac{p}{p-q}} - \bigg(\int_{x_{k+1}}^{\infty} u\bigg)^{\frac{p}{p-q}} \bigg] \esup_{s\in (0, x_k)}  \bigg(\int_0^s w \bigg)^{-\frac{q}{p-q}} V_r(0,s)^{\frac{pq}{p-q}} \\
	&\quad + \sum_{k= -\infty}^{M-2} \int_{[x_k,x_{k+1})} \bigg(\int_t^{\infty} u\bigg)^{\frac{p}{p-q}} d\bigg[\esup_{s\in (0, t)} \bigg(\int_0^s w \bigg)^{-\frac{q}{p-q}} V_r(0,s)^{\frac{pq}{p-q}} \bigg]\\
	& \quad + \int_{x_{M-1}}^{\infty} \bigg( \int_t^{\infty} u\bigg)^{\frac{q}{p-q}} u(t)  \esup_{s\in (0, t)}  \bigg( \int_0^s w \bigg)^{-\frac{q}{p-q}} V_r(0,s)^{\frac{pq}{p-q}} dt\\
& \approx \sum_{k= -\infty}^{M-2} \bigg[ \bigg(\int_{x_k}^{\infty} u\bigg)^{\frac{p}{p-q}} - \bigg(\int_{x_{k+1}}^{\infty} u\bigg)^{\frac{p}{p-q}} \bigg] \esup_{s\in (0, x_k)}  \bigg(\int_0^s w \bigg)^{-\frac{q}{p-q}} V_r(0,s)^{\frac{pq}{p-q}} \\
    & \quad  + \sum_{k= -\infty}^{M-2} \int_{[x_k,x_{k+1})} \bigg(\int_t^{x_{k+1}} u\bigg)^{\frac{p}{p-q}} d\bigg[\esup_{s\in (0, t)}  \bigg(\int_0^s w \bigg)^{-\frac{q}{p-q}} V_r(0,s)^{\frac{pq}{p-q}} \bigg]  \\
	    &\quad + \sum_{k= -\infty}^{M-1}  \bigg(\int_{x_{k}}^{\infty} u\bigg)^{\frac{p}{p-q}} \bigg[ \esup_{s\in (0, x_{k})}  \bigg(\int_0^s w \bigg)^{-\frac{q}{p-q}} V_r(0,s)^{\frac{pq}{p-q}}\\
			&\hspace{5.5cm} -\esup_{s\in (0, x_{k-1})}  \bigg(\int_0^s w \bigg)^{-\frac{q}{p-q}} V_r(0,s)^{\frac{pq}{p-q}} \bigg]\\
			& \quad + \int_{x_{M-1}}^{\infty} \bigg( \int_t^{\infty} u\bigg)^{\frac{q}{p-q}} u(t)  \esup_{s\in (0, t)}  \bigg( \int_0^s w \bigg)^{-\frac{q}{p-q}} V_r(0,s)^{\frac{pq}{p-q}} dt.
\end{align*}
One more use of integration by parts tells us that
\begin{align}\label{C_4-1}
C_4^{\frac{pq}{p-q}} & \lesssim
\sum_{k= -\infty}^{M-2} \bigg[ \bigg(\int_{x_k}^{\infty} u\bigg)^{\frac{p}{p-q}} - \bigg(\int_{x_{k+1}}^{\infty} u\bigg)^{\frac{p}{p-q}} \bigg] \esup_{s\in (0, x_k)}  \bigg(\int_0^s w \bigg)^{-\frac{q}{p-q}} V_r(0,s)^{\frac{pq}{p-q}} \\
    & \quad  + \sum_{k= -\infty}^{M-1} \int_{x_k}^{x_{k+1}} \bigg(\int_t^{x_{k+1}} u\bigg)^{\frac{q}{p-q}} u(t) \esup_{s\in (0, t)}  \bigg(\int_0^s w \bigg)^{-\frac{q}{p-q}} V_r(0,s)^{\frac{pq}{p-q}} dt  \notag\\
	    &\quad + \sum_{k= -\infty}^{M-1}  \bigg(\int_{x_{k}}^{\infty} u\bigg)^{\frac{p}{p-q}} \bigg[ \esup_{s\in (0, x_{k})}  \bigg(\int_0^s w \bigg)^{-\frac{q}{p-q}} V_r(0,s)^{\frac{pq}{p-q}} \notag\\
			&\hspace{5.5cm} -\esup_{s\in (0, x_{k-1})}  \bigg(\int_0^s w \bigg)^{-\frac{q}{p-q}} V_r(0,s)^{\frac{pq}{p-q}} \bigg].\notag
\end{align}
Applying \eqref{Abel} with
\begin{equation*}
c_k = \int_{x_k}^{x_{k+1}} \bigg(\int_t^{\infty} u\bigg)^{\frac{q}{p-q}} u(t)dt
\end{equation*}
and
\begin{equation}\label{bk}
    b_k = \esup_{s\in (0,x_k)} \bigg(\int_0^s w \bigg)^{-\frac{q}{p-q}} V_r(0,s)^{\frac{pq}{p-q}},
\end{equation}
we can see that
\begin{align*}
&\sum_{k= -\infty}^{M-2} \bigg[ \bigg(\int_{x_k}^{\infty} u\bigg)^{\frac{p}{p-q}} - \bigg(\int_{x_{k+1}}^{\infty} u\bigg)^{\frac{p}{p-q}} \bigg] \esup_{s\in (0, x_k)}  \bigg(\int_0^s w \bigg)^{-\frac{q}{p-q}} V_r(0,s)^{\frac{pq}{p-q}} \\
& \quad = \sum_{k= -\infty}^{M-2} \bigg[\int_{x_k}^{x_{k+1}} \bigg(\int_t^{\infty} u\bigg)^{\frac{q}{p-q}} u(t)dt\bigg] \esup_{s\in (0, x_k)}  \bigg(\int_0^s w \bigg)^{-\frac{q}{p-q}} V_r(0,s)^{\frac{pq}{p-q}} \\
&\quad  \approx \sum_{k= -\infty}^{M-2}  \bigg(\int_{x_{k}}^{x_{M-1}} u\bigg)^{\frac{p}{p-q}} \bigg[ \esup_{s\in (0, x_{k})}  \bigg(\int_0^s w \bigg)^{-\frac{q}{p-q}} V_r(0,s)^{\frac{pq}{p-q}}\\
			&\hspace{5.5cm} -\esup_{s\in (0, x_{k-1})}  \bigg(\int_0^s w \bigg)^{-\frac{q}{p-q}} V_r(0,s)^{\frac{pq}{p-q}} \bigg]\\
		&\quad  + \lim_{k \rightarrow -\infty} \bigg(\int_{x_k}^{x_{M-1}} u\bigg)^{\frac{p}{p-q}} \esup_{s\in (0, x_{k})}  \bigg(\int_0^s w \bigg)^{-\frac{q}{p-q}} V_r(0,s)^{\frac{pq}{p-q}}.
\end{align*}
Plugging this in \eqref{C_4-1}, we obtain
\begin{align*}
C_4^{\frac{pq}{p-q}} & \lesssim
\sum_{k= -\infty}^{M-1} \int_{x_k}^{x_{k+1}} \bigg(\int_t^{x_{k+1}} u\bigg)^{\frac{q}{p-q}} u(t) \esup_{s\in (0, t)}  \bigg(\int_0^s w \bigg)^{-\frac{q}{p-q}} V_r(0,s)^{\frac{pq}{p-q}} dt  \notag\\
	    &\quad + \sum_{k= -\infty}^{M-1}  \bigg(\int_{x_{k}}^{\infty} u\bigg)^{\frac{p}{p-q}} \bigg[ \esup_{s\in (0, x_{k})}  \bigg(\int_0^s w \bigg)^{-\frac{q}{p-q}} V_r(0,s)^{\frac{pq}{p-q}} \notag\\
			&\hspace{5.5cm} -\esup_{s\in (0, x_{k-1})}  \bigg(\int_0^s w \bigg)^{-\frac{q}{p-q}} V_r(0,s)^{\frac{pq}{p-q}} \bigg]\\
			& \quad + \lim_{k \rightarrow -\infty} \bigg(\int_{x_k}^{\infty} u\bigg)^{\frac{p}{p-q}} \esup_{s\in (0, x_{k})}  \bigg(\int_0^s w \bigg)^{-\frac{q}{p-q}} V_r(0,s)^{\frac{pq}{p-q}}\\
		& = C_{4,1} + C_{4,2} + C_{4,3}.
\end{align*}
Next, \eqref{difference-u} with $s=\frac{q}{p-q}$, $b_k$ as in \eqref{bk} and $a_k = \int_{x_k}^{x_{k+1}} u$ combined with \eqref{A3-equiv} give
\begin{align} \label{C43-A1}
C_{4,2} + C_{4,3} & \approx  \sum_{k= -\infty}^{M-1}  \bigg(\int_{x_k}^{x_{k+1}} u\bigg)\bigg(\int_{x_k}^{\infty} u\bigg)^{\frac{q}{p-q}} \esup_{s\in (0, x_k)}  \bigg(\int_0^s w \bigg)^{-\frac{q}{p-q}} V_r(0,s)^{\frac{pq}{p-q}} \notag \\
&\approx A_3^{\frac{pq}{p-q}}.
\end{align}
We shall now find a suitable upper estimate for $C_{4,1}$. Observe that
\begin{align*}
C_{4,1}  &\approx \sum_{k= -\infty}^{M-1} \int_{x_k}^{x_{k+1}} \bigg(\int_t^{x_{k+1}} u\bigg)^{\frac{q}{p-q}} u(t) dt \esup_{s\in (0, x_k)}  \bigg(\int_0^s w \bigg)^{-\frac{q}{p-q}} V_r(0,s)^{\frac{pq}{p-q}} \\
&\quad + \sum_{k= -\infty}^{M-1} \int_{x_k}^{x_{k+1}} \bigg(\int_t^{x_{k+1}} u\bigg)^{\frac{q}{p-q}} u(t) \esup_{s\in (x_k, t)}  \bigg(\int_0^s w \bigg)^{-\frac{q}{p-q}} V_r(0,s)^{\frac{pq}{p-q}} dt.
\end{align*}
Using the fact that $\{x_k\}_{k=-\infty}^M$ is a discretizing sequence of $W$, we get
\begin{align*}
C_{4,1}  &\lesssim  \sum_{k= -\infty}^{M-1}  \bigg(\int_{x_k}^{x_{k+1}} u\bigg)^{\frac{p}{p-q}} \esup_{s\in (0, x_k)}  \bigg(\int_0^s w \bigg)^{-\frac{q}{p-q}} V_r(0,s)^{\frac{pq}{p-q}} \\
& \quad + \sum_{k= -\infty}^{M-1} 2^{-k\frac{q}{p-q}} \int_{x_k}^{x_{k+1}} \bigg(\int_t^{x_{k+1}} u\bigg)^{\frac{q}{p-q}} u(t) V_r(0,t)^{\frac{pq}{p-q}} dt \\
& \approx  \sum_{k= -\infty}^{M-1}  \bigg(\int_{x_k}^{x_{k+1}} u\bigg)^{\frac{p}{p-q}} \esup_{s\in (0, x_k)}  \bigg(\int_0^s w \bigg)^{-\frac{q}{p-q}} V_r(0,s)^{\frac{pq}{p-q}}  \\
& \quad + \sum_{k= -\infty}^{M-1} 2^{-k\frac{q}{p-q}} \int_{x_k}^{x_{k+1}} \bigg(\int_t^{x_{k+1}} u\bigg)^{\frac{q}{p-q}} u(t) V_r(x_k,t)^{\frac{pq}{p-q}} dt \\
& \quad + \sum_{k= -\infty}^{M-1} 2^{-k\frac{q}{p-q}} V_r(0,x_k)^{\frac{pq}{p-q}} \bigg(\int_{x_k}^{x_{k+1}} u\bigg)^{\frac{p}{p-q}}.
\end{align*}
Note that we applied \eqref{V-cut} to obtain the last equivalence. We further have
\begin{align*}
C_{4,1}  &\lesssim \sum_{k= -\infty}^{M-1}  \bigg(\int_{x_k}^{x_{k+1}} u\bigg)\bigg(\int_{x_k}^{\infty} u\bigg)^{\frac{q}{p-q}} \esup_{s\in (0, x_k)}  \bigg(\int_0^s w \bigg)^{-\frac{q}{p-q}} V_r(0,s)^{\frac{pq}{p-q}} \bigg)^{\frac{p-q}{pq}} \\
& \quad + \sum_{k= -\infty}^{M-1} 2^{-k\frac{q}{p-q}} \int_{x_k}^{x_{k+1}} \bigg(\int_t^{x_{k+1}} u\bigg)^{\frac{q}{p-q}} u(t) V_r(x_k,t)^{\frac{pq}{p-q}} dt \\
& \quad + \sum_{k= -\infty}^{M-1}  \bigg(\int_{x_k}^{x_{k+1}} u\bigg)\bigg(\int_{x_k}^{\infty} u\bigg)^{\frac{q}{p-q}} \sup_{i \leq k} 2^{-i\frac{q}{p-q}} V_r(0,x_i)^{\frac{pq}{p-q}}.
\end{align*}
In view of \eqref{sup-dec} and \eqref{A3-equiv}, we obtain that
\begin{align*}
C_{4,1} & \lesssim A_3^{\frac{p-q}{pq}} +  \sum_{k= -\infty}^{M-1} 2^{-k\frac{q}{p-q}} \int_{x_k}^{x_{k+1}} \bigg(\int_t^{x_{k+1}} u\bigg)^{\frac{q}{p-q}} u(t) V_r(x_k,t)^{\frac{pq}{p-q}} dt.
\end{align*}
Moreover, since $V_r(x_k, t)$ is an increasing function and  $\frac{p-q}{p(1-q)} <1$, applying a special case of \cite[Proposition 2.1]{He-St:93}, we get
\begin{align*}
    &\int_{x_k}^{x_{k+1}} \bigg(\int_t^{x_{k+1}} u\bigg)^{\frac{q}{p-q}} u(t) V_r(x_k, t)^{\frac{pq}{p-q}} dt
        \\
    & \hspace{3cm}\lesssim  \bigg(\int_{x_k}^{x_{k+1}} \bigg(\int_t^{x_{k+1}} u\bigg)^{\frac{q}{1-q}} u(t) V_r(x_k,t)^{\frac{q}{1-q}} dt \bigg)^{\frac{p(1-q)}{p-q}}.
\end{align*}
Therefore, $C_{4,1}^{\frac{pq}{p-q}} \lesssim A_3 + B_3$. Consequently, \eqref{C43-A1} implies
\begin{equation}\label{C4<A3+B3}
    C_4 \lesssim A_3 + B_3.
\end{equation}
Next, we will prove that $C_5 \lesssim A_3 + B_3$. Assume that $\max\{A_3, B_3\} < \infty$. Then
\begin{align*}
C_5^{\frac{pq}{p-q}} & \lesssim \sum_{k=-\infty}^{M-2} \int_{x_k}^{x_{k+1}} \bigg(\int_0^x w\bigg)^{-\frac{p}{p-q}} w(x) \bigg( \int_0^{x} \bigg( \int_t^{x} u \bigg)^{\frac{q}{1-q}} u(t) V_r(0,t)^{\frac{q}{1-q}} dt \bigg)^{\frac{p(1-q)}{p-q}} dx  \notag\\
& \hspace{1cm} +  \bigg(\int_0^{\infty}  w\bigg)^{-\frac{q}{p-q}} \bigg( \int_0^{\infty} \bigg( \int_t^{\infty} u \bigg)^{\frac{q}{1-q}} u(t) V_r(0,t)^{\frac{q}{1-q}} dt \bigg)^{\frac{p(1-q)}{p-q}} \notag\\
& \lesssim \sum_{k=-\infty}^{M-2}  2^{-k\frac{q}{p-q}} \bigg( \int_0^{{x_{k+1}}} \bigg( \int_t^{{x_{k+1}}} u \bigg)^{\frac{q}{1-q}} u(t) V_r(0,t)^{\frac{q}{1-q}} dt \bigg)^{\frac{p(1-q)}{p-q}}  \notag \\
& \hspace{1cm} +  \bigg(\int_0^{\infty}  w\bigg)^{-\frac{q}{p-q}} \bigg( \int_0^{\infty} \bigg( \int_t^{\infty} u \bigg)^{\frac{q}{1-q}} u(t) V_r(0,t)^{\frac{q}{1-q}} dt \bigg)^{\frac{p(1-q)}{p-q}}.
\end{align*}
On the other hand, \eqref{A1<A3} together with \eqref{A1-equiv} yields
\begin{equation}\label{lim-0+}
    \lim_{t\rightarrow 0+} \bigg(\int_t^{\infty} u\bigg)^{\frac{1}{1-q}} V_r(0,t)^{\frac{q}{1-q}} = 0.
\end{equation}
Since $V_r(0, t) \approx V_r(0, x_{M-1}) + V_r(x_{M-1}, t)$ for each $t \in (x_{M-1}, \infty)$,  we have
\begin{equation*}
\int_{x_{M-1}}^{\infty} \bigg( \int_s^{\infty} u \bigg)^{\frac{q}{1-q}} u(s) V_r(0,s)^{\frac{q}{1-q}} ds < \infty.
\end{equation*}
Consequently, for each  $t \in (x_{M-1}, \infty)$,
\begin{equation*}
V_r(0,t)^{\frac{q}{1-q}} \bigg( \int_t^{\infty} u \bigg)^{\frac{1}{1-q}} \lesssim \bigg( \int_{t}^{\infty} \bigg( \int_s^{\infty} u \bigg)^{\frac{q}{1-q}} u(s) V_r(0,s)^{\frac{q}{1-q}} ds \bigg)^{\frac{1-q}{q}},
\end{equation*}
and one immediately obtains
\begin{equation}\label{lim-infty}
\lim_{t\rightarrow \infty} \bigg(\int_t^{\infty} u\bigg)^{\frac{1}{1-q}} V_r(0,t)^{\frac{q}{1-q}} = 0.
\end{equation}
Thus, integration by parts yields that
\begin{align}
C_5^{\frac{pq}{p-q}} & \lesssim  \sum_{k=-\infty}^{M-1}  2^{-k\frac{q}{p-q}} \bigg( \int_0^{{x_{k+1}}} \bigg( \int_t^{{x_{k+1}}} u \bigg)^{\frac{1}{1-q}} d\big[ V_r(0,t)^{\frac{q}{1-q}} \big] \bigg)^{\frac{p(1-q)}{p-q}} = D_1. \label{C5<D1}
\end{align}
We have,
\begin{align*}
D_1 & =  \sum_{k=-\infty}^{M-1}  2^{-k\frac{q}{p-q}} \bigg( \sum_{i=-\infty}^k \int_{[x_i,x_{i+1})} \bigg( \int_t^{{x_{k+1}}} u \bigg)^{\frac{1}{1-q}} d\big[ V_r(0,t)^{\frac{q}{1-q}} \big] \bigg)^{\frac{p(1-q)}{p-q}} \notag\\
& \approx \sum_{k=-\infty}^{M-1}  2^{-k\frac{q}{p-q}} \bigg( \sum_{i=-\infty}^k \int_{[x_i,x_{i+1})} \bigg( \int_t^{{x_{i+1}}} u \bigg)^{\frac{1}{1-q}} d\big[ V_r(0,t)^{\frac{q}{1-q}} \big] \bigg)^{\frac{p(1-q)}{p-q}}  \notag\\
    & \quad + \sum_{k=-\infty}^{M-1}  2^{-k\frac{q}{p-q}} \bigg( \sum_{i=-\infty}^{k-1}  \bigg( \sum_{j=i}^{k-1} \int_{x_{j+1}}^{{x_{j+2}}} u \bigg)^{\frac{1}{1-q}}  \int_{[x_i,x_{i+1})} d\big[ V_r(0,t)^{\frac{q}{1-q}} \big] \bigg)^{\frac{p(1-q)}{p-q}}.
\end{align*}
Applying \eqref{dec.sum-sum} to the first term and Minkowski's inequality with $\frac{1}{1-q}>1$ to the second one, we obtain that
\begin{align*}
D_1 &  \lesssim \sum_{k=-\infty}^{M-1}  2^{-k\frac{q}{p-q}} \bigg( \int_{[x_k, x_{k+1})} \bigg( \int_t^{{x_{k+1}}} u \bigg)^{\frac{1}{1-q}} d\big[ V_r(0,t)^{\frac{q}{1-q}} \big] \bigg)^{\frac{p(1-q)}{p-q}}  \\
    & \quad + \sum_{k=-\infty}^{M-1}  2^{-k\frac{q}{p-q}} \bigg( \sum_{j=-\infty}^{k-1}  \bigg( \int_{x_{j+1}}^{{x_{j+2}}} u \bigg) \bigg( \int_0^{x_{j+1}} d\big[ V_r(0,t)^{\frac{q}{1-q}} \big]\bigg)^{1-q} \bigg)^{\frac{p}{p-q}}.
\end{align*}
Integration by parts together with \eqref{dec.sum-sum} gives
\begin{align*}
D_1 &  \lesssim \sum_{k=-\infty}^{M-1}  2^{-k\frac{q}{p-q}} \bigg( \int_{x_k}^{{x_{k+1}}} \bigg( \int_t^{{x_{k+1}}} u \bigg)^{\frac{q}{1-q}} u(t) V_r(0,t)^{\frac{q}{1-q}} \bigg)^{\frac{p(1-q)}{p-q}} \notag \\
    & \quad + \sum_{k=-\infty}^{M-1}  2^{-k\frac{q}{p-q}} \bigg( \int_{x_{k}}^{{x_{k+1}}} u \bigg)^{\frac{p}{p-q}} V_r(0,x_k)^{\frac{pq}{p-q}}.
\end{align*}
Then, applying \eqref{V-cut} and \eqref{sup-dec}, we arrive at
\begin{align}
D_1 &\lesssim  B_3^{\frac{pq}{p-q}} + \sum_{k=-\infty}^{M-1}  2^{-k\frac{q}{p-q}} \bigg( \int_{x_{k}}^{{x_{k+1}}} u \bigg)^{\frac{p}{p-q}} V_r(0,x_k)^{\frac{pq}{p-q}} \notag\\
& \lesssim B_3^{\frac{pq}{p-q}} + \sum_{k=-\infty}^{M-1}   \bigg( \int_{x_{k}}^{{x_{k+1}}} u \bigg) \bigg( \int_{x_{k}}^{\infty} u \bigg)^{\frac{q}{p-q}} \sup_{i \leq k} 2^{-i\frac{q}{p-q}} V_r(0,x_i)^{\frac{pq}{p-q}} \notag \\
& \approx B_3^{\frac{pq}{p-q}} + A_3^{\frac{pq}{p-q}} \label{D1<A3+B3}.
\end{align}
Finally, we arrive at $C_4 + C_5 \lesssim A_3 + B_3$. It remains to prove that $A_3 + B_3 \lesssim C_4 + C_5$. Putting $a_k = \int_{x_k}^{x_{k+1}} u$, $s=\frac{q}{p-q}$  and
$$
    b_k = \sup_{i\leq k}  2^{-i\frac{q}{p-q}} V_r(x_{i-1},x_i)^{\frac{pq}{p-q}}
$$
in \eqref{difference-u}, we have
\begin{align*}
A_3^{\frac{pq}{p-q}} & \approx \sum_{k=-\infty}^{M-1} \bigg( \int_{x_k}^{\infty} u\bigg)^{\frac{p}{p-q}} \\
	& \quad \times \Big[ \sup_{i\leq k}  2^{-i\frac{q}{p-q}} V_r(x_{i-1},x_i)^{\frac{pq}{p-q}} - \sup_{i\leq k-1}  2^{-i\frac{q}{p-q}} V_r(x_{i-1},x_i)^{\frac{pq}{p-q}}\Big]   \\
	&\quad +  \lim_{k \rightarrow -\infty} \bigg(\int_{x_k}^{\infty} u\bigg)^{\frac{p}{p-q}}  \sup_{i\leq k}  2^{-i\frac{q}{p-q}}  V_r(x_{i-1},x_i)^{\frac{pq}{p-q}}\\
& \approx \sum_{k=-\infty}^{M-1}  \bigg( \sum_{i=k} ^{M-1}  \int_{x_i}^{x_{i+1}} \bigg(\int_{s}^{\infty} u\bigg)^{\frac{q}{p-q}} u(s) ds \bigg)  \\
	& \quad \times \Big[ \sup_{i\leq k}  2^{-i\frac{q}{p-q}} V_r(x_{i-1},x_i)^{\frac{pq}{p-q}} - \sup_{i\leq k-1}  2^{-i\frac{q}{p-q}} V_r(x_{i-1},x_i)^{\frac{pq}{p-q}}\Big]  \\
	&+  \lim_{k \rightarrow -\infty}  \bigg(\sum_{i=k} ^{M-1}  \int_{x_i}^{x_{i+1}} \bigg(\int_{s}^{\infty} u\bigg)^{\frac{q}{p-q}} u(s) ds\bigg) \sup_{i\leq k}  2^{-i\frac{q}{p-q}} V_r(x_{i-1},x_i)^{\frac{pq}{p-q}}.
\end{align*}
Now, taking
$$
c_k =  \int_{x_k}^{x_{k+1}} \bigg(\int_{s}^{\infty} u\bigg)^{\frac{q}{p-q}} u(s) ds
$$
in \eqref{Abel} and applying \eqref{A3-equiv},  we obtain
\begin{align*}
A_3^{\frac{pq}{p-q}} & \approx \sum_{k=-\infty}^{M-1}  \bigg( \int_{x_k}^{x_{k+1}} \bigg(\int_{s}^{\infty} u\bigg)^{\frac{q}{p-q}} u(s) ds \bigg) \sup_{i\leq k}  2^{-i\frac{q}{p-q}} V_r(x_{i-1},x_i)^{\frac{pq}{p-q}}\\
	& \approx \sum_{k=-\infty}^{M-1}  \bigg( \int_{x_k}^{x_{k+1}} \bigg(\int_{s}^{\infty} u\bigg)^{\frac{q}{p-q}} u(s) ds \bigg) 	\esup_{t \in (0, x_k)} \bigg(\int_0^t w\bigg)^{-\frac{q}{p-q}} V_r(0,t)^{\frac{pq}{p-q}}\\
& \leq \sum_{k=-\infty}^{M-1} \int_{x_k}^{x_{k+1}} \bigg(\int_t^{\infty} u\bigg)^{\frac{q}{p-q}} u(t)   \esup_{s\in (0, t)}  \bigg(\int_0^s w \bigg)^{-\frac{q}{p-q}} V_r(0,s)^{\frac{pq}{p-q}} dt\\
&= C_4^{\frac{pq}{p-q}}.
\end{align*}
On the other hand, the definition of the discretizing sequence $\{x_k\}_{k=-\infty}^{M}$ yields that
\begin{align}
B_3^{\frac{pq}{p-q}} & \leq \sum_{k=-\infty}^{M-1} 2^{-k\frac{q}{p-q}} \bigg( \int_{0}^{x_{k+1}} \bigg( \int_t^{x_{k+1}} u \bigg)^{\frac{q}{1-q}} u(t) V_r(0,t)^{\frac{q}{1-q}} dt \bigg)^{\frac{p(1-q)}{p-q}} \notag\\
& \approx  \sum_{k=-\infty}^{M-2} \int_{x_{k+1}}^{x_{k+2}} \bigg(\int_0^x w \bigg)^{-\frac{p}{p-q}} w(x) dx \bigg( \int_{0}^{x_{k+1}} \bigg( \int_t^{x_{k+1}} u \bigg)^{\frac{q}{1-q}} u(t) V_r(0,t)^{\frac{q}{1-q}} dt \bigg)^{\frac{p(1-q)}{p-q}}  \notag \\
& \quad + 2^{-M\frac{q}{p-q}} \bigg( \int_{0}^{x_{M}} \bigg( \int_t^{x_{M}} u \bigg)^{\frac{q}{1-q}} u(t) V_r(0,t)^{\frac{q}{1-q}} dt \bigg)^{\frac{p(1-q)}{p-q}} \notag \\
& \lesssim C_5^{\frac{pq}{p-q}}. \label{B3<C5}
\end{align}
Consequently, we have $A_3 + B_3 \lesssim C_4 + C_5$, which completes the proof in this case.
\medskip

\rm{(vi)} By Theorem~\ref{C:discrete solutions}(vi), we have that $C \approx A_4 + B_3$. We will prove that $ A_4+B_3  \approx  C_5 + C_6 $.

First of all, it is clear that $A_3 \leq A_4$. Then, using \eqref{C5<D1} combined with \eqref{D1<A3+B3}, we obtain that $C_5 \lesssim   B_3+ A_3 \leq B_3 + A_4$. We shall find an upper estimate for  $C_6$.	Denote by
\begin{equation*}
\Phi(x,y) = \bigg( \int_x^y \bigg(\int_0^s w \bigg)^{-\frac{p}{p-r}} w(s) V_r(0,s)^{\frac{pr}{p-r}} ds \bigg)^{\frac{q(p-r)}{r(p-q)}}.
\end{equation*}
Then, we have
\begin{align*}
C_6^{\frac{pq}{p-q}} & = \sum_{i=-\infty}^{M-1} \int_{x_k}^{x_{k+1}} \bigg(\int_0^x w\bigg)^{-2}w(x)  \sup_{y \in (0, x)} \bigg(\int_0^y w \bigg) \bigg(\int_y^x u(s)\bigg(\int_s^{\infty} u \bigg)^{\frac{q}{p-q}}ds\bigg) \Phi(0,y)  dx \\
& \lesssim  \sum_{k=-\infty}^{M-1} 2^{-k} \sup_{y \in (0, x_{k+1})} \bigg(\int_0^y w \bigg)
\bigg(\int_y^{x_{k+1}} u(t)\bigg(\int_t^{\infty} u \bigg)^{\frac{q}{p-q}} dt\bigg) \Phi(0,y)\\
& = \sum_{k=-\infty}^{M-1} 2^{-k} \sup_{i \leq k}
\sup_{y \in (x_i, x_{i+1})} \bigg(\int_0^y w \bigg)
\bigg(\int_y^{x_{k+1}} u(t) \bigg(\int_t^{\infty} u \bigg)^{\frac{q}{p-q}} dt\bigg) \Phi(0,y)\\
&\approx \sum_{k=-\infty}^{M-1} 2^{-k} \sup_{i \leq k} 2^{i}
\sup_{y \in (x_i, x_{i+1})}
\bigg(\int_y^{x_{k+1}} u(t)\bigg(\int_t^{\infty} u \bigg)^{\frac{q}{p-q}} dt\bigg) \Phi(0,y) \\
&\approx  \sum_{k=-\infty}^{M-1} 2^{-k} \sup_{i \leq k} 2^{i} \sup_{y \in (x_i, x_{i+1})}
\bigg(\int_y^{x_{i+1}} u(t)\bigg(\int_t^{\infty} u \bigg)^{\frac{q}{p-q}} dt\bigg) \Phi(0,y) \\
& \quad +  \sum_{k=-\infty}^{M-1} 2^{-k} \sup_{i \leq k-1} 2^{i}
\bigg(\int_{x_{i+1}}^{x_{k+1}} u(t)\bigg(\int_t^{\infty} u \bigg)^{\frac{q}{p-q}} dt\bigg) \Phi(0,x_{i+1}).
\end{align*}
Since $\Phi(0, y) \approx \Phi(0, x_i)+ \Phi(x_i, x_{i+1})$ for every $y \in (x_i, x_{i+1})$, we have
\begin{align*}
C_6^{\frac{pq}{p-q}} &\lesssim  \sum_{k=-\infty}^{M-1} 2^{-k} \sup_{i \leq k} 2^{i}
\bigg(\int_{x_i}^{x_{i+1}} u(t)\bigg(\int_t^{\infty} u \bigg)^{\frac{q}{p-q}} dt\bigg) \Phi(0,x_i) \\
&\quad +  \sum_{k=-\infty}^{M-1} 2^{-k} \sup_{i \leq k} 2^{i} \sup_{y \in (x_i, x_{i+1})} \bigg(\int_y^{x_{i+1}} u(t)\bigg(\int_t^{\infty} u \bigg)^{\frac{q}{p-q}} dt\bigg) \Phi(x_i,y)\\
& \quad +  \sum_{k=-\infty}^{M-1} 2^{-k} \sup_{i \leq k-1} 2^{i} \bigg(\int_{x_{i+1}}^{x_{k+1}} u(t)\bigg(\int_t^{\infty} u \bigg)^{\frac{q}{p-q}} dt\bigg) \Phi(0,x_{i+1}).
\end{align*}
Then, using the fact that $\{2^i \Phi(0,x_{i+1})\}$ is a strongly increasing sequence and applying \eqref{inc.sup-sum} to the third term, we obtain that
\begin{align*}
C_6^{\frac{pq}{p-q}} & \lesssim  \sum_{k=-\infty}^{M-1} 2^{-k} \sup_{i \leq k} 2^{i}
\bigg(\int_{x_i}^{x_{i+1}} u(t)\bigg(\int_t^{\infty} u \bigg)^{\frac{q}{p-q}} dt\bigg) \Phi(0,x_i)\\
& \quad + \sum_{k=-\infty}^{M-1} 2^{-k} \sup_{i \leq k} 2^{i} \sup_{y \in (x_i, x_{i+1})}
\bigg(\int_y^{x_{i+1}} u(t)\bigg(\int_t^{\infty} u \bigg)^{\frac{q}{p-q}} dt\bigg) \Phi(x_i,y).
\end{align*}
Moreover, \eqref{dec.sum-sup} yields
\begin{align}
C_6^{\frac{pq}{p-q}} &\lesssim  \sum_{k=-\infty}^{M-1} \bigg(\int_{x_k}^{x_{k+1}} u(t)\bigg(\int_t^{\infty} u \bigg)^{\frac{q}{p-q}} dt\bigg) \Phi(0,x_k) \notag \\
&\quad +  \sum_{k=-\infty}^{M-1} \sup_{y \in (x_k, x_{k+1})}
\bigg(\int_y^{x_{k+1}} u(t)\bigg(\int_t^{\infty} u \bigg)^{\frac{q}{p-q}} dt\bigg) \Phi(x_k,y) \notag\\
& = C_{6,1} + C_{6,2}. \label{C6<C61+C62}
\end{align}
Applying \eqref{Abel} with $b_k = \Phi(0,x_k)$ and
\begin{equation}\label{C:6-ck}
c_k = \int_{x_k}^{x_{k+1}} u(t)\bigg(\int_t^{\infty} u \bigg)^{\frac{q}{p-q}} dt,
\end{equation}
we have that
\begin{align*}
C_{6,1} &\approx \sum_{k=-\infty}^{M-1} \bigg(\int_{x_k}^{\infty} u \bigg)^{\frac{p}{p-q}} \big[\Phi(0,x_k) - \Phi(0,x_{k-1}) \big].
\end{align*}
Moreover, taking $s= \frac{q}{p-q}$, $a_k = \int_{x_k}^{x_{k+1}} u$ and $b_k = \Phi(0, x_k)$, \eqref{difference-u} gives
\begin{align}\label{C61-upper}
C_{6,1} &\approx  \sum_{k=-\infty}^{M-1} \bigg(\int_{x_k}^{x_{k+1}} u \bigg)\bigg(\int_{x_k}^{\infty} u \bigg)^{\frac{q}{p-q}} \Phi(0,x_k).
\end{align}
Finally, \eqref{lower-A2} together with \eqref{C61-upper} yield that
\begin{equation}\label{C61<A4}
    C_{6,1} \lesssim A_4^{\frac{pq}{p-q}}.
\end{equation}
Let us proceed with $C_{6,2}$. It is easy to see that
\begin{align*}
C_{6,2}  \lesssim  \sum_{k=-\infty}^{M-1} 2^{-k\frac{q}{p-q}} \esup_{y \in (x_k, x_{k+1})}
\bigg(\int_y^{x_{k+1}} u(t)\bigg(\int_t^{\infty} u \bigg)^{\frac{q}{p-q}} dt\bigg) V_r(0,y)^{\frac{pq}{p-q}}.
\end{align*}
Let $y_k \in [x_k, x_{k+1}]$, $k \leq M-1$ be such that
\begin{align}\label{sup-yk}
&\esup_{y \in (x_k, x_{k+1})}
\bigg(\int_y^{x_{k+1}} u(t)\bigg(\int_t^{\infty} u \bigg)^{\frac{q}{p-q}} dt\bigg)  V_r(0,y)^{\frac{pq}{p-q}} \notag \\
& \hspace{3cm}\lesssim
\bigg(\int_{y_k}^{x_{k+1}} u(t)\bigg(\int_t^{\infty} u \bigg)^{\frac{q}{p-q}} dt\bigg) V_r(0,y_k)^{\frac{pq}{p-q}}.
\end{align}
Observe that
$$
2^k \approx \int_0^{x_k} w \leq \int_0^{y_k} w \leq \int_0^{x_{k+1}} w \approx 2^{k+1} \quad\text{for $k \leq M-1$.}
$$
Thus, $\{y_k\}_{k=-\infty}^{M-1}$ is a discretizing sequence of $W$. Then, using \eqref{sup-yk}, we get that
\begin{align*}
C_{6,2} & \lesssim  \sum_{k=-\infty}^{M-1} 2^{-k\frac{q}{p-q}} \bigg(\int_{y_k}^{x_{k+1}} u(t)\bigg(\int_t^{\infty} u \bigg)^{\frac{q}{p-q}} dt\bigg) V_r(0,y_k)^{\frac{pq}{p-q}}  \\
& \leq \sum_{k=-\infty}^{M-1} 2^{-k\frac{q}{p-q}} \bigg(\int_{y_k}^{y_{k+1}} u(t)\bigg(\int_t^{\infty} u \bigg)^{\frac{q}{p-q}} dt\bigg) V_r(0,y_k)^{\frac{pq}{p-q}} \\
& \leq \sum_{k=-\infty}^{M-1}  \bigg(\int_{y_k}^{y_{k+1}} u(t)\bigg(\int_t^{\infty} u \bigg)^{\frac{q}{p-q}} dt\bigg) \bigg(\sum_{i=-\infty}^k 2^{-i\frac{r}{p-r}} V_r(0,y_i)^{\frac{pr}{p-r}}\bigg)^{\frac{q(p-r)}{r(p-q)}}.
\end{align*}
Applying \eqref{Abel} with
$$
c_k = \int_{y_k}^{y_{k+1}} u(t) \bigg(\int_t^\infty u\bigg)^{\frac{q}{p-q}} dt
$$
and
$$
b_k = \bigg(\sum_{i=-\infty}^k 2^{-i\frac{r}{p-r}} V_r(0,y_i)^{\frac{pr}{p-r}}\bigg)^{\frac{q(p-r)}{r(p-q)}},
$$
and then using \eqref{difference-u} with $a_k = \int_{y_k}^{y_{k+1}} u$, we get
\begin{align*}
C_{6,2} & \lesssim  \sum_{k=-\infty}^{M-1} \bigg(\int_{y_k}^{\infty} u \bigg)^{\frac{p}{p-q}}  \bigg[\bigg(\sum_{i=-\infty}^k 2^{-i\frac{r}{p-r}} V_r(0,y_i)^{\frac{pr}{p-r}}\bigg)^{\frac{q(p-r)}{r(p-q)}} - \\
& \hspace{5cm} - \bigg(\sum_{i=-\infty}^{k-1} 2^{-i\frac{r}{p-r}} V_r(0,y_i)^{\frac{pr}{p-r}}\bigg)^{\frac{q(p-r)}{r(p-q)}} \bigg]   \\
& \approx  \sum_{k=-\infty}^{M-1} \bigg(\int_{y_k}^{y_{k+1}} u \bigg)\bigg(\int_{y_k}^{\infty} u \bigg)^{\frac{q}{p-q}}  \bigg(\sum_{i=-\infty}^k 2^{-i\frac{r}{p-r}} V_r(0,y_i)^{\frac{pr}{p-r}}\bigg)^{\frac{q(p-r)}{r(p-q)}}.
\end{align*}
Since $\{y_k\}_{k=-\infty}^{M-1}$ is a discretizing sequence of $W$, we obtain from \eqref{A2-equiv}
\begin{equation}\label{C62<A4}
    C_{6,2} \lesssim A_4^{\frac{pq}{p-q}}.
\end{equation}
Therefore, in view of \eqref{C6<C61+C62}, \eqref{C61<A4} and \eqref{C62<A4}, we arrive at $C_6 \lesssim A_4$.
Consequently, we obtain that $ C_5 + C_6  \lesssim B_3 + A_4$.
It remains to show that $B_3 + A_4 \lesssim C_5 + C_6$. We have proved in \eqref{B3<C5} that $B_3 \lesssim C_5$. Taking $s=\frac{q}{p-q}$, $a_k = \int_{x_k}^{x_{k+1}} u$ and
$$
b_k = \bigg(\sum_{i=-\infty}^k 2^{-i\frac{r}{p-r}} V_r(0,x_i)^{\frac{pr}{p-r}}\bigg)^{\frac{q(p-r)}{r(p-q)}}
$$
in \eqref{difference-u} and additionally, using \eqref{upper-A2}, we have
\begin{align*}
    A_4^{\frac{pq}{p-q}} & \approx \sum_{k=-\infty}^{M-1} \bigg(\int_{x_k}^{\infty} u \bigg)^{\frac{p}{p-q}}  \bigg[\bigg(\sum_{i=-\infty}^k 2^{-i\frac{r}{p-r}} V_r(0,x_i)^{\frac{pr}{p-r}}\bigg)^{\frac{q(p-r)}{r(p-q)}} -
        \\
    & \hspace{5cm} - \bigg(\sum_{i=-\infty}^{k-1} 2^{-i\frac{r}{p-r}} V_r(0,x_i)^{\frac{pr}{p-r}}\bigg)^{\frac{q(p-r)}{r(p-q)}} \bigg]
        \\
    &\approx
    \sum_{k=-\infty}^{M-2} \bigg(\int_{x_{k+1}}^{\infty} u \bigg)^{\frac{p}{p-q}}  \bigg[\bigg(\sum_{i=-\infty}^k 2^{-i\frac{r}{p-r}} V_r(0,x_i)^{\frac{pr}{p-r}}\bigg)^{\frac{q(p-r)}{r(p-q)}} -
        \\
    & \hspace{5cm} - \bigg(\sum_{i=-\infty}^{k-1} 2^{-i\frac{r}{p-r}} V_r(0,x_i)^{\frac{pr}{p-r}}\bigg)^{\frac{q(p-r)}{r(p-q)}} \bigg]
        \\
    &\hspace{1cm} + \sum_{k=-\infty}^{M-1} \bigg(\int_{x_{k}}^{x_{k+1}} u \bigg)^{\frac{p}{p-q}}  \bigg[\bigg(\sum_{i=-\infty}^k 2^{-i\frac{r}{p-r}} V_r(0,x_i)^{\frac{pr}{p-r}}\bigg)^{\frac{q(p-r)}{r(p-q)}} -
        \\
    & \hspace{5cm} - \bigg(\sum_{i=-\infty}^{k-1} 2^{-i\frac{r}{p-r}} V_r(0,x_i)^{\frac{pr}{p-r}}\bigg)^{\frac{q(p-r)}{r(p-q)}} \bigg]
        \\
    &\lesssim
    \sum_{k=-\infty}^{M-2} \bigg(\int_{x_{k+1}}^{x_{k+2}} u(s) \bigg(\int_{s}^{\infty} u \bigg)^{\frac{q}{p-q}}dt\bigg) \bigg(\sum_{i=-\infty}^k 2^{-i\frac{r}{p-r}} V_r(0,x_i)^{\frac{pr}{p-r}}\bigg)^{\frac{q(p-r)}{r(p-q)}}
        \\
    &\hspace{1cm} + \sum_{k=-\infty}^{M-1} \bigg(\int_{x_{k}}^{x_{k+1}} u \bigg)^{\frac{p}{p-q}}  \bigg(\sum_{i=-\infty}^{k-1} 2^{-i\frac{r}{p-r}} V_r(0,x_i)^{\frac{pr}{p-r}}\bigg)^{\frac{q(p-r)}{r(p-q)}}
        \\
    &\hspace{1cm} + \sum_{k=-\infty}^{M-1} 2^{-k\frac{q}{p-q}}\bigg(\int_{x_{k}}^{x_{k+1}} u \bigg)^{\frac{p}{p-q}}   V_r(0,x_k)^{\frac{pq}{p-q}}
        \\
    &\lesssim
    \sum_{k=-\infty}^{M-1} \bigg(\int_{x_{k}}^{x_{k+1}} u(s) \bigg(\int_{s}^{\infty} u \bigg)^{\frac{q}{p-q}}dt\bigg)  \bigg(\sum_{i=-\infty}^{k-1} 2^{-i\frac{r}{p-r}} V_r(0,x_i)^{\frac{pr}{p-r}}\bigg)^{\frac{q(p-r)}{r(p-q)}}
        \\
    &\hspace{1cm} + \sum_{k=-\infty}^{M-1} 2^{-k\frac{q}{p-q}}\bigg(\int_{x_{k}}^{x_{k+1}} u \bigg)^{\frac{p}{p-q}}   V_r(0,x_k)^{\frac{pq}{p-q}}
        \\
    & \lesssim \sum_{k=-\infty}^{M-1} \bigg(\int_{x_{k}}^{x_{k+1}} u(s) \bigg(\int_{s}^{\infty} u \bigg)^{\frac{q}{p-q}}dt\bigg)  \bigg( \int_{0}^{x_{k}} \bigg(\int_0^t w \bigg)^{-\frac{p}{p-r}}  w(t) V_r(0, t)^{\frac{pr}{p-r}}\bigg)^{\frac{q(p-r)}{r(p-q)}}
        \\
    &\hspace{1cm} + \sum_{k=-\infty}^{M-1} 2^{-k\frac{q}{p-q}}\bigg(\int_{x_{k}}^{x_{k+1}} u \bigg)^{\frac{p}{p-q}}   V_r(0,x_k)^{\frac{pq}{p-q}}
        \\
    &\lesssim
     \sum_{k=-\infty}^{M-2} \bigg(\int_{x_k}^{x_{k+1}} u(s) \bigg(\int_{s}^{\infty} u \bigg)^{\frac{q}{p-q}}
    ds\bigg) \bigg( \int_{0}^{x_{k}} \bigg(\int_0^t w \bigg)^{-\frac{p}{p-r}}  w(t) V_r(0, t)^{\frac{pr}{p-r}}\bigg)^{\frac{q(p-r)}{r(p-q)}}
        \\
    &\quad  +\sum_{k=-\infty}^{M-2}  2^{-k\frac{q}{p-q}}\bigg(\int_{x_k}^{x_{k+1}} u \bigg)^{\frac{p}{p-q}}
      V_r(0,x_k)^{\frac{pq}{p-q}}
        \\
    &\quad  + \bigg(\int_0^\infty w\bigg)^{-\frac{q}{p-q}}\bigg(\int_{x_{M-1}}^{\infty} u(s) \bigg(\int_{s}^{\infty} u \bigg)^{\frac{q}{p-q}}
    ds \bigg)  V_r(0,x_{M-1})^{\frac{pq}{p-q}},
\end{align*}
that is,
\begin{equation}\label{A4<J1+J2+J3}
    A_4 \lesssim J_1 + J_2 + J_3.
\end{equation}
Observe that
\begin{align}\label{J1<C6}
J_1 & \approx  \sum_{k=-\infty}^{M-2} \int_{x_{k+1}}^{x_{k+2}} \bigg(\int_0^x w \bigg)^{-2} w(x) dx \notag\\
    & \hspace{3cm} \times \bigg(\int_0^{x_k} w\bigg)  \bigg(\int_{x_k}^{x_{k+1}} u(s) \bigg(\int_{s}^{\infty} u \bigg)^{\frac{q}{p-q}}
ds\bigg) \Phi(0,x_k)\notag\\
& \leq \sum_{k=-\infty}^{M-2} \int_{x_{k+1}}^{x_{k+2}} \bigg(\int_0^x w \bigg)^{-2} w(x) dx \notag\\
    & \hspace{3cm} \times \sup_{y \in (0, x_{k+1})} \bigg(\int_0^{y} w\bigg)  \bigg(\int_{y}^{x_{k+1}} u(s) \bigg(\int_{s}^{\infty} u \bigg)^{\frac{q}{p-q}}
ds\bigg) \Phi(0,y) \notag\\
& \leq C_6^{\frac{pq}{p-q}}.
\end{align}
Next, it is easy to see that
\begin{align}\label{J2<C5}
J_2 & \approx \sum_{k=-\infty}^{M-2}  2^{-k\frac{q}{p-q}}\bigg(\int_{x_k}^{x_{k+1}} u(s) \bigg(\int_{s}^{x_{k+1}} u \bigg)^{\frac{q}{1-q}}
ds \bigg)^{\frac{p(1-q)}{p-q}}  V_r(0,x_k)^{\frac{pq}{p-q}}  \notag\\
&\lesssim  \sum_{k=-\infty}^{M-2} \int_{x_{k+1}}^{x_{k+2}} \bigg(\int_0^x w\bigg)^{-\frac{p}{p-q}} w(x) dx
 \bigg(\int_{x_k}^{x_{k+1}} u(s) \bigg(\int_{s}^{x_{k+1}} u \bigg)^{\frac{q}{1-q}} V_r(0, s)^{\frac{q}{1-q}}
ds \bigg)^{\frac{p(1-q)}{p-q}}\notag  \\
& \leq C_5^{\frac{pq}{p-q}}.
\end{align}
And, finally,
\begin{align}\label{J3<C5}
J_3 & \approx \bigg(\int_0^{\infty} w \bigg)^{-\frac{q}{p-q}}  \bigg(\int_{x_{M-1}}^{\infty} u(s) \bigg(\int_{s}^{\infty} u \bigg)^{\frac{q}{1-q}}
ds \bigg)^{\frac{p(1-q)}{p-q}} V_r(0, x_{M-1})^{\frac{pq}{p-q}}  \leq C_5^{\frac{pq}{p-q}}.
\end{align}
Altogether, we arrive at $B_3 + A_4 \lesssim C_5 + C_6$, establishing the result in this case.
\medskip

\rm{(vii)} By Theorem~\ref{C:discrete solutions}(vii), we have that $C \approx A_4 + B_4$. We will prove that $ A_4+B_4  \approx C_6 + C_7$.

Let us start with the estimate $C_6 + C_7 \lesssim A_4 + B_4$. In the proof of the case (vi) we have shown that $C_6 \lesssim A_4$. On the other hand,
\begin{align*}
C_7 & \approx  \bigg(\sum_{k=-\infty}^{M-1} \int_{x_k}^{x_{k+1}} \bigg(\int_0^x w\bigg)^{-\frac{p}{p-q}} w(x) \esup_{t \in (0,x)} \bigg( \int_t^{\infty} u \bigg)^{\frac{p}{p-q}} V_r(0,t) ^{\frac{pq}{p-q}} dx \bigg)^{\frac{p-q}{pq}}\\
& \hspace{1cm}+ \bigg(\int_0^{x_M}  w\bigg)^{-\frac{1}{p}} \esup_{t \in (0, x_M)} \bigg( \int_t^{x_M} u \bigg)^{\frac{1}{q}} V_r(0, t)\\
& \lesssim \bigg(\sum_{k=-\infty}^{M-1} 2^{-k\frac{q}{p-q}}\esup_{t \in (0,x_{k+1})} \bigg( \int_t^{\infty} u \bigg)^{\frac{p}{p-q}} V_r(0,t) ^{\frac{pq}{p-q}} \bigg)^{\frac{p-q}{pq}}\\
& = \bigg(\sum_{k=-\infty}^{M-1} 2^{-k\frac{q}{p-q}} \sup_{i \leq k} \esup_{t \in (x_i,x_{i+1})} \bigg( \int_t^{\infty} u \bigg)^{\frac{p}{p-q}} V_r(0,t) ^{\frac{pq}{p-q}} \bigg)^{\frac{p-q}{pq}} .
\end{align*}
Using \eqref{dec.sum-sup} and applying \eqref{V-cut}, we have
\begin{align*}
C_7^{\frac{pq}{p-q}}
& \lesssim \sum_{k=-\infty}^{M-1} 2^{-k\frac{q}{p-q}}\esup_{t \in (x_k,x_{k+1})} \bigg( \int_t^{\infty} u \bigg)^{\frac{p}{p-q}} V_r(0,t) ^{\frac{pq}{p-q}}\\
&\approx \sum_{k=-\infty}^{M-1} 2^{-k\frac{q}{p-q}} \bigg( \int_{x_k}^{\infty} u \bigg)^{\frac{p}{p-q}} V_r(0,x_k) ^{\frac{pq}{p-q}} + \sum_{k=-\infty}^{M-2} 2^{-k\frac{q}{p-q}} \bigg( \int_{x_{k+1}}^{\infty} u \bigg)^{\frac{p}{p-q}} V_r(x_k,x_{k+1}) ^{\frac{pq}{p-q}} \\
& \qquad + \sum_{k=-\infty}^{M-1} 2^{-k\frac{q}{p-q}}\esup_{t \in (x_k,x_{k+1})} \bigg( \int_t^{x_{k+1}} u \bigg)^{\frac{p}{p-q}} V_r(x_k,t) ^{\frac{pq}{p-q}}.
\end{align*}
Note that, since $\{2^{-\frac{k}{p}} \big( \int_{x_k}^{\infty} u \big)^{\frac{1}{q}}\}$ is strongly decreasing and in view of \eqref{Vr(0,xk)},
\begin{equation}\label{C8-upper}
\sum_{k=-\infty}^{M-1} 2^{-k\frac{q}{p-q}} \bigg( \int_{x_k}^{\infty} u \bigg)^{\frac{p}{p-q}} V_r(0,x_k) ^{\frac{pq}{p-q}} \approx \sum_{k=-\infty}^{M-1} 2^{-k\frac{q}{p-q}} \bigg( \int_{x_k}^{\infty} u \bigg)^{\frac{p}{p-q}} V_r(x_{k-1},x_k)^{\frac{pq}{p-q}}
\end{equation}
holds if we apply \eqref{dec.sup-sum} when $r < 1$ and  \eqref{dec.sup-sup} when $r=1$.
Then, using \eqref{C8-upper}, we obtain that
\begin{align}
C_7^{\frac{pq}{p-q}} & \lesssim  \sum_{k=-\infty}^{M-1} 2^{-k\frac{q}{p-q}} \bigg( \int_{x_k}^{\infty} u \bigg)^{\frac{p}{p-q}} V_r(x_{k-1},x_k) ^{\frac{pq}{p-q}} \notag\\
&\quad  + \sum_{k=-\infty}^{M-1} 2^{-k\frac{q}{p-q}}\esup_{t \in (x_k,x_{k+1})} \bigg( \int_t^{x_{k+1}} u \bigg)^{\frac{p}{p-q}} V_r(x_k,t) ^{\frac{pq}{p-q}}\notag\\
& = \sum_{k=-\infty}^{M-1} 2^{-k\frac{q}{p-q}} \bigg( \int_{x_k}^{\infty} u \bigg)^{\frac{p}{p-q}} V_r(x_{k-1},x_k) ^{\frac{pq}{p-q}}  + B_4^{\frac{pq}{p-q}}\notag\\
&= D_2 + B_4^{\frac{pq}{p-q}}. \label{D2}
\end{align}
Using the fact that
\begin{equation*}
    \bigg(\int_{x_k}^{\infty} u\bigg)^{\frac{p}{p-q}} \approx \sum_{i= k}^{M-1} \int_{x_i}^{x_{i+1}} u(s) \bigg(\int_{s}^{\infty} u\bigg)^{\frac{q}{p-q}} ds,
\end{equation*}
we get
\begin{align*}
D_2 \approx  \sum_{k=-\infty}^{M-1} 2^{-k\frac{q}{p-q}}  V_r(x_{k-1},x_k)^{\frac{pq}{p-q}} \sum_{i= k}^{M-1} \int_{x_i}^{x_{i+1}} u(s) \bigg(\int_{s}^{\infty} u\bigg)^{\frac{q}{p-q}} ds.
\end{align*}
Interchanging the order of sums gives
\begin{align*}
D_2 \approx \sum_{k=-\infty}^{M-1} \bigg(\int_{x_k}^{x_{k+1}} u(s) \bigg(\int_{s}^{\infty} u\bigg)^{\frac{q}{p-q}} ds\bigg) \sum_{i= -\infty}^k   2^{-i\frac{q}{p-q}}  V_r(x_{i-1},x_i)^{\frac{pq}{p-q}}.
\end{align*}
Since $\frac{q(p-r)}{r(p-q)}>1$, we have
\begin{align*}
\sum_{i= -\infty}^k  2^{-i\frac{q}{p-q}} V_r(x_{i-1}, x_i)^{\frac{pq}{p-q}} \leq \bigg( \sum_{i=-\infty}^{k}  2^{-i\frac{r}{p-r}} V_r(x_{i-1},x_i)^{\frac{pr}{p-r}} \bigg)^{\frac{q(p-r)}{r(p-q)}}.
\end{align*}
Then, first using  \eqref{Abel}  with $c_k$ as in \eqref{C:6-ck} and
$$
b_k = \bigg(\sum_{i=-\infty}^k 2^{-i\frac{r}{p-r}} V_r(x_{i-1}, x_i)^{\frac{pr}{p-r}}\bigg)^{\frac{q(p-r)}{r(p-q)}},
$$
then applying \eqref{difference-u} with $a_k = \int_{x_k}^{x_{k+1}} u$ and $b_k$ as above, we obtain \begin{align}
D_2 &\lesssim \sum_{k=-\infty}^{M-1} \bigg(\int_{x_k}^{x_{k+1}} u(s) \bigg(\int_{s}^{\infty} u\bigg)^{\frac{q}{p-q}} ds\bigg) \bigg( \sum_{i=-\infty}^{k}  2^{-i\frac{r}{p-r}} V_r(x_{i-1},x_i)^{\frac{pr}{p-r}} \bigg)^{\frac{q(p-r)}{r(p-q)}} \notag \\
& \approx \sum_{k=-\infty}^{M-1} \bigg(\int_{x_k}^{\infty} u \bigg)^{\frac{p}{p-q}}  \bigg[\bigg(\sum_{i=-\infty}^k 2^{-i\frac{r}{p-r}} V_r(x_{i-1},x_i)^{\frac{pr}{p-r}}\bigg)^{\frac{q(p-r)}{r(p-q)}} - \notag \\
& \hspace{5cm} - \bigg(\sum_{i=-\infty}^{k-1} 2^{-i\frac{r}{p-r}} V_r(x_{i-1},x_i)^{\frac{pr}{p-r}}\bigg)^{\frac{q(p-r)}{r(p-q)}} \bigg] \notag\\
& \approx A_4^{\frac{pq}{p-q}}.\label{D2<A4}
\end{align}
Thus, using \eqref{D2}, we arrive at $C_7 \lesssim  B_4 + A_4$. It remains to show that $B_4 + A_4 \lesssim C_6 + C_7$. It is clear that
\begin{align}\label{J2+J3<C7}
J_2^{\frac{p-q}{pq}}+J_3^{\frac{p-q}{pq}}&\lesssim  \sum_{k=-\infty}^{M-1} \bigg(\int_{x_k}^{x_{k+1}} w(s)\bigg(\int_0^{s} w \bigg)^{-\frac{p}{p-q}}ds \bigg)  \esup_{t\in(0,x_{k})}\bigg(\int_{t}^{\infty} u \bigg)^{\frac{p}{p-q}}
  V_r(0,t)^{\frac{pq}{p-q}}  \notag\\
	&\lesssim C_7.
	\end{align}
We have from \eqref{A4<J1+J2+J3} together with \eqref{J1<C6} and \eqref{J2+J3<C7} that $A_4 \lesssim C_6 + C_7$. Furthermore,
\begin{align*}
B_4^{\frac{pq}{p-q}} & \leq \sum_{k=-\infty}^{M-1} 2^{-k\frac{q}{p-q}}\esup_{t \in (x_k,x_{k+1})} \bigg( \int_t^{\infty} u \bigg)^{\frac{p}{p-q}} V_r(0,t) ^{\frac{pq}{p-q}} \\
& \approx \sum_{k=-\infty}^{M-2} \int_{x_{k+1}}^{x_{k+2}} \bigg(\int_0^x w\bigg)^{-\frac{p}{p-q}} w(x) dx \esup_{t \in (x_k,x_{k+1})} \bigg( \int_t^{\infty} u \bigg)^{\frac{p}{p-q}} V_r(0,t) ^{\frac{pq}{p-q}} \\
& +  \bigg(\int_0^{\infty} w\bigg)^{-\frac{q}{p-q}} \esup_{t \in (x_{M-1}, \infty)} \bigg( \int_t^{\infty} u \bigg)^{\frac{p}{p-q}} V_r(0,t)^{\frac{pq}{p-q}}\\
& \leq C_7^{\frac{pq}{p-q}}.
\end{align*}
Consequently, we arrive at $A_4 + B_4 \lesssim C_6 + C_7 \lesssim A_4 + B_4$, which completes the proof.
\end{proof}

\begin{proof}[Proof of Theorem \ref{T:equiv.solut.}] By Theorem~\ref{C:discrete solutions}(vi), we have $C \approx A_4 + B_3$. We will prove that $ A_4+B_3  \approx  \mathcal{C}_5 + \mathcal{C}_6$. We know from \eqref{B3<C5} that $B_3 \lesssim C_5$. Also, it is clear that $C_5 \lesssim \mathcal{C}_5$. Then $B_3 \lesssim \mathcal{C}_5$. On the other hand, it is easy to see that $J_1 \leq  \mathcal{C}_6^{\frac{pq}{p-q}}$. Hence, the combination of   \eqref{A4<J1+J2+J3}, \eqref{J2<C5} and \eqref{J3<C5} leads to $A_4 \lesssim \mathcal{C}_5 + \mathcal{C}_6$. Thus, we get $B_3 + A_4 \lesssim \mathcal{C}_5 + \mathcal{C}_6$.

Now, we will prove that $\mathcal{C}_5 + \mathcal{C}_6 \lesssim B_3 + A_4$. One has
\begin{align}
\mathcal{C}_5^{\frac{pq}{p-q}} & \approx \sum_{k=-\infty}^{M-1} \int_{x_k}^{x_{k+1}} \bigg(\int_0^x w\bigg)^{-\frac{p}{p-q}} w(x) \bigg( \int_0^{x} \bigg( \int_t^{\infty} u \bigg)^{\frac{q}{1-q}} u(t) V_r(0,t)^{\frac{q}{1-q}} dt \bigg)^{\frac{p(1-q)}{p-q}} dx  \notag\\
& \hspace{1cm} +  \bigg(\int_0^{x_M}  w\bigg)^{-\frac{q}{q-p}} \bigg( \int_0^{x_M} \bigg( \int_t^{x_M} u \bigg)^{\frac{q}{1-q}} u(t) V_r(0,t)^{\frac{q}{1-q}} dt \bigg)^{\frac{p(1-q)}{p-q}} \notag\\
& \lesssim \sum_{k=-\infty}^{M-1} 2^{-k\frac{q}{p-q}} \bigg( \int_0^{x_{k+1}} \bigg( \int_t^{\infty} u \bigg)^{\frac{q}{1-q}} u(t) V_r(0,t)^{\frac{q}{1-q}} dt \bigg)^{\frac{p(1-q)}{p-q}} \notag\\
& \hspace{1cm} +  \bigg(\int_0^{x_M}  w\bigg)^{-\frac{q}{q-p}} \bigg( \int_0^{x_M} \bigg( \int_t^{x_M} u \bigg)^{\frac{q}{1-q}} u(t) V_r(0,t)^{\frac{q}{1-q}} dt \bigg)^{\frac{p(1-q)}{p-q}} \notag\\
& \lesssim \sum_{k=-\infty}^{M-1} 2^{-k\frac{q}{p-q}} \bigg( \int_0^{x_{k+1}} \bigg( \int_t^{\infty} u \bigg)^{\frac{q}{1-q}} u(t) V_r(0,t)^{\frac{q}{1-q}} dt \bigg)^{\frac{p(1-q)}{p-q}} = D. \label{D}
\end{align}
Since $\max\{A_4, B_3\}<\infty$ and $A_3 \leq A_4$, observe that \eqref{lim-0+} and \eqref{lim-infty} hold in this case as well. Therefore, integration by parts and \eqref{C8-upper} give
\begin{align*}
D  &  \lesssim \sum_{k=-\infty}^{M-1} 2^{-k\frac{q}{p-q}} \bigg( \int_0^{x_{k+1}} \bigg( \int_t^{\infty} u \bigg)^{\frac{1}{1-q}} d\big[V_r(0,t)^{\frac{q}{1-q}}\big] dt \bigg)^{\frac{p(1-q)}{p-q}}  \\
&  \approx \sum_{k=-\infty}^{M-1} 2^{-k\frac{q}{p-q}} \bigg( \int_{0}^{x_{k+1}} \bigg( \int_t^{x_{k+1}} u \bigg)^{\frac{1}{1-q}} d\big[V_r(0,t)^{\frac{q}{1-q}}\big] dt \bigg)^{\frac{p(1-q)}{p-q}}  \\
& \quad + \sum_{k=-\infty}^{M-1} 2^{-k\frac{q}{p-q}} \bigg( \int_{x_k}^{\infty} u \bigg)^{\frac{p}{p-q}} V_r(0,x_k)^{\frac{pq}{p-q}}\\
&\approx   D_1 + D_2,
\end{align*}
where $D_1$ and $D_2$ are defined in \eqref{C5<D1} and \eqref{D2}, respectively.
Therefore, using the fact that $A_3 \leq A_4$ together with \eqref{D1<A3+B3}, we have $D_1 \lesssim A_4^{\frac{p-q}{pq}} + B_3^{\frac{p-q}{pq}}$. Moreover, since $\frac{q(p-r)}{r(p-q)}\geq1$ holds in this case we also have by  \eqref{D2<A4} that $D_2 \lesssim A_4^{\frac{p-q}{pq}}$. Thus, in view of \eqref{D}, we have
\begin{equation*}
    D \lesssim B_3^{\frac{pq}{p-q}} +A_4^{\frac{pq}{p-q}},
\end{equation*}
which yields that $\mathcal{C}_5  \lesssim B_3 +A_4$. Similarly,
\begin{align*}
\mathcal{C}_6^{\frac{pq}{p-q}} &= \sum_{k=\-\infty}^{M-1}\int_{x_k}^{x_{k+1}} \bigg(\int_t^{\infty} u\bigg)^{\frac{q}{p-q}} u(t)  \bigg(\int_0^t \bigg(\int_0^s w \bigg)^{-\frac{p}{p-r}} w(s) V_r(0, s)^{\frac{pr}{p-r}} ds \bigg)^{\frac{q(p-r)}{r(p-q)}} dt \\
&\approx \sum_{k=\-\infty}^{M-1}\int_{x_k}^{x_{k+1}} \bigg(\int_t^{\infty} u\bigg)^{\frac{q}{p-q}} u(t)  \bigg(\int_{x_k}^t \bigg(\int_0^s w \bigg)^{-\frac{p}{p-r}} w(s) V_r(0, s)^{\frac{pr}{p-r}} ds \bigg)^{\frac{q(p-r)}{r(p-q)}} dt \\
&\quad + \sum_{k=\-\infty}^{M-1}\int_{x_k}^{x_{k+1}} \bigg(\int_t^{\infty} u\bigg)^{\frac{q}{p-q}} u(t)  dt
\bigg(\int_0^{x_k} \bigg(\int_0^s w \bigg)^{-\frac{p}{p-r}} w(s) V_r(0, s)^{\frac{pr}{p-r}} ds \bigg)^{\frac{q(p-r)}{r(p-q)}}
\\
&\lesssim \sum_{k=\-\infty}^{M-1}\int_{x_k}^{x_{k+1}} \bigg(\int_t^{\infty} u\bigg)^{\frac{q}{p-q}} u(t)  V_r(0, t)^{\frac{pq}{p-q}}dt  \bigg(\int_{x_k}^{x_{k+1}} \bigg(\int_0^s w \bigg)^{-\frac{p}{p-r}} w(s)ds \bigg)^{\frac{q(p-r)}{r(p-q)}} \\
&+ \sum_{k=\-\infty}^{M-1}\int_{x_k}^{x_{k+1}} \bigg(\int_t^{\infty} u\bigg)^{\frac{q}{p-q}} u(t)  dt
\bigg(\int_0^{x_k} \bigg(\int_0^s w \bigg)^{-\frac{p}{p-r}} w(s) V_r(0, s)^{\frac{pr}{p-r}} ds \bigg)^{\frac{q(p-r)}{r(p-q)}}
\\
& \lesssim \sum_{k=\-\infty}^{M-1}\int_{x_k}^{x_{k+1}} \bigg(\int_t^{\infty} u\bigg)^{\frac{q}{p-q}} u(t)  V_r(0, t)^{\frac{pq}{p-q}} \bigg(\int_{0}^{t} w(s)ds \bigg)^{-\frac{q}{p-q}}dt   \\
&+ \sum_{k=\-\infty}^{M-1}\int_{x_k}^{x_{k+1}} \bigg(\int_t^{\infty} u\bigg)^{\frac{q}{p-q}} u(t)  dt
\bigg(\int_0^{x_k} \bigg(\int_0^s w \bigg)^{-\frac{p}{p-r}} w(s) V_r(0, s)^{\frac{pr}{p-r}} ds \bigg)^{\frac{q(p-r)}{r(p-q)}}
\\
& \lesssim C_4^{\frac{pq}{p-q}} + A_4^{\frac{pq}{p-q}}.
\end{align*}
In the last inequality we applied \eqref{C61<A4} to the second term. Then, using \eqref{C4<A3+B3} and the fact that $A_3\leq A_4$, we arrive at  $\mathcal{C}_6 \lesssim A_3 + B_3 + A_4 \leq A_4 + B_3$, and the proof is complete.
\end{proof}

\noindent\textbf{Acknowledgment.} The authors gladly acknowledge stimulating discussions with Mar\'{\i}a Carro which proved very useful especially in the initial stages of the work on this project.

\bibliographystyle{abbrv}

\begin{thebibliography}{10}
	
	\bibitem{As-Ma:14}
	S.~V. Astashkin and L.~Maligranda.
	\newblock Structure of {C}es\`aro function spaces: a survey.
	\newblock In {\em Function spaces {X}}, volume 102 of {\em Banach Center
		Publ.}, pages 13--40. Polish Acad. Sci. Inst. Math., Warsaw, 2014.
	
	\bibitem{Ba-Ku@86}
	R.~J. Bagby and D.~S. Kurtz.
	\newblock A rearranged good {$\lambda$} inequality.
	\newblock {\em Trans. Amer. Math. Soc.}, 293(1):71--81, 1986.
	
	\bibitem{Ba-Mi-Ru:03}
	J.~Bastero, M.~Milman, and F.~J. Ruiz~Blasco.
	\newblock A note on {$L(\infty,q)$} spaces and {S}obolev embeddings.
	\newblock {\em Indiana Univ. Math. J.}, 52(5):1215--1230, 2003.
	
	\bibitem{Be-De-Sh:81}
	C.~Bennett, R.~A. DeVore, and R.~Sharpley.
	\newblock Weak-{$L^{\infty }$} and {BMO}.
	\newblock {\em Ann. of Math. (2)}, 113(3):601--611, 1981.
	
	\bibitem{Be-Sh:88}
	C.~Bennett and R.~Sharpley.
	\newblock {\em Interpolation of operators}, volume 129 of {\em Pure and Applied
		Mathematics}.
	\newblock Academic Press, Inc., Boston, MA, 1988.
	
	\bibitem{Be:91}
	G.~Bennett.
	\newblock Some elementary inequalities. {III}.
	\newblock {\em Quart. J. Math. Oxford Ser. (2)}, 42(166):149--174, 1991.
	
	\bibitem{Be:96}
	G.~Bennett.
	\newblock Factorizing the classical inequalities.
	\newblock {\em Mem. Amer. Math. Soc.}, 120(576):viii+130, 1996.
	
	\bibitem{Be-Gr:05}
	G.~Bennett and K.-G. Grosse-Erdmann.
	\newblock On series of positive terms.
	\newblock {\em Houston J. Math.}, 31(2):541--586, 2005.
	
	\bibitem{Bo:70}
	R.~P. Boas, Jr.
	\newblock Some integral inequalities related to {H}ardy's inequality.
	\newblock {\em J. Analyse Math.}, 23:53--63, 1970.
	
	\bibitem{Bo-Ma:05}
	S.~Boza and J.~Mart\'{\i}n.
	\newblock Equivalent expressions for norms in classical {L}orentz spaces.
	\newblock {\em Forum Math.}, 17(3):361--373, 2005.
	
	\bibitem{Bo-So:11}
	S.~Boza and J.~Soria.
	\newblock Solution to a conjecture on the norm of the {H}ardy operator minus
	the identity.
	\newblock {\em J. Funct. Anal.}, 260(4):1020--1028, 2011.
	
	\bibitem{Bo-So:19}
	S.~Boza and J.~Soria.
	\newblock Averaging operators on decreasing or positive functions: equivalence
	and optimal bounds.
	\newblock {\em J. Approx. Theory}, 237:135--152, 2019.
	
	\bibitem{Bo-So:20}
	S.~Boza and J.~Soria.
	\newblock Weak-type and end-point norm estimates for {H}ardy operators.
	\newblock {\em Ann. Mat. Pura Appl. (4)}, 199(6):2381--2393, 2020.
	
	\bibitem{Ca-Go-Ma-Pi:08}
	M.~Carro, A.~Gogatishvili, J.~Mart\'{i}n, and L.~Pick.
	\newblock Weighted inequalities involving two {H}ardy operators with
	applications to embeddings of function spaces.
	\newblock {\em J. Operator Theory}, 59(2):309--332, 2008.
	
	\bibitem{Ca-Go-Ma-Pi:05}
	M.~J. Carro, A.~Gogatishvili, J.~Mart\'{\i}n, and L.~Pick.
	\newblock Functional properties of rearrangement invariant spaces defined in
	terms of oscillations.
	\newblock {\em J. Funct. Anal.}, 229(2):375--404, 2005.
	
	\bibitem{Ca-So:93}
	M.~J. Carro and J.~Soria.
	\newblock Boundedness of some integral operators.
	\newblock {\em Canad. J. Math.}, 45(6):1155--1166, 1993.
	
	\bibitem{Ed-Mi-Mu-Pi:20}
	D.~E. Edmunds, Z.~Mihula, V.~Musil, and L.~Pick.
	\newblock Boundedness of classical operators on rearrangement-invariant spaces.
	\newblock {\em J. Funct. Anal.}, 278(4):108341, 56, 2020.
	
	\bibitem{Ev-Go-Op:08}
	W.~D. Evans, A.~Gogatishvili, and B.~Opic.
	\newblock {\em Weighted inequalities involving {$\rho$}-quasiconcave
		operators}.
	\newblock World Scientific Publishing Co. Pte. Ltd., Hackensack, NJ, 2018.
	
	\bibitem{Fr:97}
	M.~Franciosi.
	\newblock A condition implying boundedness and {VMO} for a function {$f$}.
	\newblock {\em Studia Math.}, 123(2):109--116, 1997.
	
	\bibitem{Go-Mi-Pi-Tu-Un:21}
	A.~Gogatishvili, Z.~Mihula, L.~Pick, H.~Tur{\v{c}}inov\'a, and T.~\"{U}nver.
	\newblock Weighted inequalities for a superposition of the {C}opson operator
	and the {H}ardy operator.
	\newblock To appear in J. Fourier Anal. Appl. ArXiv:2109.03095, 2022.
	
	\bibitem{Go-Mu-Un:17}
	A.~Gogatishvili, R.~Mustafayev, and T.~\"{U}nver.
	\newblock Embeddings between weighted {C}opson and {C}es\`aro function spaces.
	\newblock {\em Czechoslovak Math. J.}, 67(142)(4):1105--1132, 2017.
	
	\bibitem{Go-Mu-Un:19}
	A.~Gogatishvili, R.~C. Mustafayev, and T.~\"{U}nver.
	\newblock Pointwise multipliers between weighted {C}opson and {C}es\`aro
	function spaces.
	\newblock {\em Math. Slovaca}, 69(6):1303--1328, 2019.
	
	\bibitem{Go-Pi:03}
	A.~Gogatishvili and L.~Pick.
	\newblock Discretization and anti-discretization of rearrangement-invariant
	norms.
	\newblock {\em Publ. Mat.}, 47(2):311--358, 2003.
	
	\bibitem{Go-He-St:96}
	M.~L. Gol'dman, H.~P. Heinig, and V.~D. Stepanov.
	\newblock On the principle of duality in {L}orentz spaces.
	\newblock {\em Canad. J. Math.}, 48(5):959--979, 1996.
	
	\bibitem{Gr:98}
	K.-G. Grosse-Erdmann.
	\newblock {\em The blocking technique, weighted mean operators and {H}ardy's
		inequality}, volume 1679 of {\em Lecture Notes in Mathematics}.
	\newblock Springer-Verlag, Berlin, 1998.
	
	\bibitem{He-St:93}
	H.~P. Heinig and V.~D. Stepanov.
	\newblock Weighted {H}ardy inequalities for increasing functions.
	\newblock {\em Canad. J. Math.}, 45(1):104--116, 1993.
	
	\bibitem{Iw:82}
	T.~Iwaniec.
	\newblock Extremal inequalities in {S}obolev spaces and quasiconformal
	mappings.
	\newblock {\em Z. Anal. Anwendungen}, 1(6):1--16, 1982.
	
	\bibitem{Ka:81}
	V.~Kabaila.
	\newblock On imbedding the space {$L_{p}(\mu )$} into {$L_{r}(\nu )$}.
	\newblock {\em Litovsk. Mat. Sb.}, 21(4):143--148, 1981.
	
	\bibitem{Ke-Pi:06}
	R.~Kerman and L.~Pick.
	\newblock Optimal {S}obolev imbeddings.
	\newblock {\em Forum Math.}, 18(4):535--570, 2006.
	
	\bibitem{Ke-Pi:09}
	R.~Kerman and L.~Pick.
	\newblock Optimal {S}obolev imbedding spaces.
	\newblock {\em Studia Math.}, 192(3):195--217, 2009.
	
	\bibitem{Ko:14}
	V.~I. Kolyada.
	\newblock Optimal relationships between {$L^p$}-norms for the {H}ardy operator
	and its dual.
	\newblock {\em Ann. Mat. Pura Appl. (4)}, 193(2):423--430, 2014.
	
	\bibitem{Ko:19}
	V.~I. Kolyada.
	\newblock On {C}\`esaro and {C}opson norms of nonnegative sequences.
	\newblock {\em Ukra\"{\i}n. Mat. Zh.}, 71(2):220--229, 2019.
	
	\bibitem{Ko:20}
	V.~I. Kolyada.
	\newblock On the optimal relationships between {$L^p$}-norms for the {H}ardy
	operator and its dual for decreasing functions.
	\newblock {\em J. Approx. Theory}, 252:105362, 5, 2020.
	
	\bibitem{Ko-Kr-Le:48}
	B.~I. Korenblyum, S.~G. Kre\u{\i}n, and B.~Y. Levin.
	\newblock On certain nonlinear questions of the theory of singular integrals.
	\newblock {\em Doklady Akad. Nauk SSSR (N.S.)}, 62:17--20, 1948.
	
	\bibitem{Kr-Ma-Pe:00}
	N.~Krugljak, L.~Maligranda, and L.~E. Persson.
	\newblock On an elementary approach to the fractional {H}ardy inequality.
	\newblock {\em Proc. Amer. Math. Soc.}, 128(3):727--734, 2000.
	
	\bibitem{Kr-Se:08}
	N.~Kruglyak and E.~Setterqvist.
	\newblock Sharp estimates for the identity minus {H}ardy operator on the cone
	of decreasing functions.
	\newblock {\em Proc. Amer. Math. Soc.}, 136(7):2505--2513, 2008.
	
	\bibitem{Ku-Pe-Sa:17}
	A.~Kufner, L.-E. Persson, and N.~Samko.
	\newblock {\em Weighted inequalities of {H}ardy type}.
	\newblock World Scientific Publishing Co. Pte. Ltd., Hackensack, NJ, second
	edition, 2017.
	
	\bibitem{Kr-Mi-Tu:21}
	M.~K\v{r}epela, Z.~Mihula, and H.~Tur\v{c}inov\'{a}.
	\newblock Discretization and antidiscretization of {L}orentz norms with no
	restriction on weights.
	\newblock {\em Rev.\ Mat.\ Complut.}, Online First, 2021.
	\newblock doi:\ 10.1007/s13163-021-00399-7.
	
	\bibitem{Lo:51}
	G.~G. Lorentz.
	\newblock On the theory of spaces {$\Lambda$}.
	\newblock {\em Pacific J. Math.}, 1:411--429, 1951.
	
	\bibitem{Ma-Pi:02}
	J.~Mal\'{y} and L.~Pick.
	\newblock An elementary proof of sharp {S}obolev embeddings.
	\newblock {\em Proc. Amer. Math. Soc.}, 130(2):555--563, 2002.
	
	\bibitem{Mi-Pu:04}
	M.~Milman and E.~Pustylnik.
	\newblock On sharp higher order {S}obolev embeddings.
	\newblock {\em Commun. Contemp. Math.}, 6(3):495--511, 2004.
	
	\bibitem{Pu:05}
	E.~Pustylnik.
	\newblock Sobolev type inequalities in ultrasymmetric spaces with applications
	to {O}rlicz-{S}obolev embeddings.
	\newblock {\em J. Funct. Spaces Appl.}, 3(2):183--208, 2005.
	
	\bibitem{Sh:70}
	J.-S. Shiue.
	\newblock A note on {C}es\`aro function space.
	\newblock {\em Tamkang J. Math.}, 1(2):91--95, 1970.
	
	\bibitem{Si:02}
	G.~Sinnamon.
	\newblock Embeddings of concave functions and duals of {L}orentz spaces.
	\newblock {\em Publ. Mat.}, 46(2):489--515, 2002.
	
	\bibitem{Si:03}
	G.~Sinnamon.
	\newblock Transferring monotonicity in weighted norm inequalities.
	\newblock {\em Collect. Math.}, 54(2):181--216, 2003.
	
	\bibitem{Si-St:96}
	G.~Sinnamon and V.~D. Stepanov.
	\newblock The weighted {H}ardy inequality: new proofs and the case {$p=1$}.
	\newblock {\em J. London Math. Soc. (2)}, 54(1):89--101, 1996.
	
	\bibitem{St:93}
	V.~D. Stepanov.
	\newblock The weighted {H}ardy's inequality for nonincreasing functions.
	\newblock {\em Trans. Amer. Math. Soc.}, 338(1):173--186, 1993.
	
	\bibitem{St:17}
	M.~Strzelecki.
	\newblock The {$L^p$}-norms of the {B}eurling-{A}hlfors transform on radial
	functions.
	\newblock {\em Ann. Acad. Sci. Fenn. Math.}, 42(1):73--93, 2017.
	
	\bibitem{St:20}
	M.~Strzelecki.
	\newblock Hardy's operator minus identity and power weights.
	\newblock {\em J. Funct. Anal.}, 279(2):108532, 34, 2020.
	
	\bibitem{Tu:21}
	H.~Tur{\v{c}}inov\'a.
	\newblock Basic functional properties of certain scale of
	rearrangement-invariant spaces.
	\newblock To appear in Math. Nachr.\ ArXiv:2009.05351 [math.FA], 2021.
	
	\bibitem{Un:20}
	T.~\"{U}nver.
	\newblock Embeddings between weighted {C}es\`aro function spaces.
	\newblock {\em Math. Inequal. Appl.}, 23(3):925--942, 2020.
	
	\bibitem{Un:21}
	T.~\"{U}nver Y{\i}ld{\i}z.
	\newblock Embeddings between weighted {T}andori and {C}esàro function spaces.
	\newblock {\em Commun. Fac. Sci. Univ. Ank. Ser. A1. Math. Stat.},
	70(2):837--848, 2021.
	
\end{thebibliography}

\end{document}